\documentclass[a4paper,12pt]{amsart}
\usepackage{version}
\usepackage{amssymb}
\usepackage{amsthm}
\usepackage{amsfonts}
\usepackage{graphicx}
\usepackage{tikz}
\usetikzlibrary{arrows}

\usepackage{xcolor}
\usepackage[pagebackref]{hyperref}
\hypersetup{
   colorlinks,
    linkcolor={red!60!black},
    citecolor={blue!60!black},
    urlcolor={blue!90!black}
}

\usepackage{epstopdf}
\usepackage[english]{babel}
\usepackage{float}
\usepackage{cite}
\usepackage{tikz,pgf}
\usetikzlibrary{patterns,spy}

\usepackage{geometry}
\newtheorem{theorem}{Theorem}[section]

\newtheorem{lemma}[theorem]{Lemma}
\newtheorem{corollary}[theorem]{Corollary}
\newtheorem{proposition}[theorem]{Proposition}

\newtheorem{claim}[theorem]{Claim}

\theoremstyle{definition}
\newtheorem{definition}[theorem]{Definition}

\newtheorem{question}[theorem]{Question}

\theoremstyle{remark}
\newtheorem{remark}[theorem]{Remark}

\numberwithin{equation}{section}

\renewcommand\bigskip{\medskip}

\def\to{\rightarrow}

\def\cF{\mathcal{F}}
\def\N{\mathbb N}

\def\cM{\mathcal{M}}
\def\Q{\mathbb Q}

\def\R{\mathbb R}

\def\Z{\mathbb Z}

\def\cV{\mathcal{V}}
\def\cF{\mathcal{F}}
\def\cE{\mathcal{N}_0}
\def\CC{{\rm conv}\,}
\def\cW{\mathcal{W}}
\def\cA{\mathcal{A}}

\def\cJ{\mathcal{J}}
\def\cN{\mathcal{N}}

\newcommand{\ifi}{\mathtt{a}}
\newcommand{\kfi}{\mathtt{c}}
\newcommand{\io}{\mathtt{i}}
\newcommand{\jo}{\mathtt{j}}
\newcommand{\lfi}{\mathtt{c}}
\newcommand{\jfi}{\mathtt{b}}

\title[The scenery flow of self-similar measures]{The scenery flow of self-similar measures with weak separation condition}
\author{Aleksi Pyörälä}
\address{Research Unit of Mathematical Sciences, P.O.Box 8000, FI-90014, University of Oulu, Finland}
\email{aleksi.pyorala@oulu.fi}

\begin{document}

\thanks{The research of this project has been supported by the Academy of Finland. I am thankful to my supervisors Ville Suomala and Meng Wu for their helpful comments and suggestions.}

\begin{abstract}
We show that self-similar measures on $\R^d$ satisfying the weak separation condition are uniformly scaling. Our approach combines elementary ergodic theory with geometric analysis of the structure given by the weak separation condition. 
\end{abstract}

\maketitle

\section{Introduction}
Taking tangents and obtaining information on an object by studying its small-scale structure is a classical idea in analysis. Much like tangent spaces can be used to study the local structure of differentiable manifolds, tangent measures can be used to study the local structure of a Radon measure $\mu$. Tangent measures are defined as the weak-$^*$ accumulation points of the \emph{scenery} of $\mu$ at $x$, that is, the flow $(\mu_{x,t})_{t \geq 0}$, where
$$
\mu_{x,t}(A) = \frac{\mu(e^{-t}A + x)}{\mu(B(x,e^{-t}))}
$$ 
for measurable $A \subseteq B(0,1)$, and $B(x,e^{-t})$ denotes the closed ball centered at $x$ and of radius $e^{-t}$. 

The collection of tangent measures around a point can be very large, making it difficult to infer from this collection any information on the original measure. For example, O'Neil \cite{Oneil} has shown that there exist Radon measures on $\R^d$ which possess every non-zero Radon measure as a tangent measure, in almost every point. This raises the question whether some tangent measures are more relevant than others in terms of studying the structure of the original measure.

Recently, it has been observed that instead of individual tangent measures, it is often more useful to study the statistical behaviour of the scenery $(\mu_{x,t})_{t \geq 0}$. To establish these statistics, one views the scenery as the orbit of $\mu$ under the continuous magnification operation at $x$. This idea gives rise to the notion of \emph{tangent distributions}, defined as the accumulation points of the \emph{scenery flow}
$$
\left( \frac{1}{T} \int_0^T \delta[\mu_{x, t}]\,dt \right)_{T > 0}.
$$ 
Here and throughout, $\delta[y]$ denotes the Dirac measure at the point $y$. Tangent distributions at $x$ are supported on the collection of all tangent measures at $x$, and are closely connected to Furstenberg's \emph{CP-processes} \cite{Furstenberg} and their coordinate-independendent generalizations, Hochman's \emph{fractal distributions} \cite{Hochman}. The theory of tangent distributions was greatly developed by Hochman and Shmerkin in \cite{HS} and further by Hochman in \cite{Hochman}. 

It is often possible to study geometric properties of a measure by studying the structure of its tangent distributions. In particular, measures for which the scenery flow converges, as $T \to \infty$, to a common distribution in almost every point in their support, are geometrically much more regular than arbitrary measures. For example, they are always exact dimensional \cite{Hochman}. Measures of this kind are called \emph{uniformly scaling}, a concept first introduced by Gavish \cite{Gavish}.

During recent years, certain dynamical properties of tangent distributions have proved to be a powerful tool in establishing fine-structure properties of fractal measures. For example, when the scenery flow converges to an \emph{ergodic fractal distribution} (see Definition \ref{fractaldistribution}) in almost every point, as is often the case for uniformly scaling measures due to a remarkable result of Hochman \cite{Hochman}, the measure satisfies a strong version of the classical Marstrand's projection theorem: If $\mu$ is a uniformly scaling measure on $\R^d$ generating an ergodic fractal distribution, then for every $k$ and $\varepsilon > 0$ there exists an open dense set $\mathcal{U}_\varepsilon$ in the space of linear projections from $\R^d$ to $\R^k$ such that for all $\pi \in \mathcal{U}_\varepsilon$,
\begin{equation}\label{marstrand}
\dim \pi\mu > \min \lbrace k, \dim \mu \rbrace - \varepsilon.
\end{equation}
This is due to Hochman and Shmerkin \cite[Theorem 8.2]{HS}. Here and in the sequel $\dim$ refers to Hausdorff dimension. 

In \cite{FFS}, Ferguson, Fraser and Sahlsten applied this projection theorem to prove Falconer's distance set conjecture in some special cases: If $\mu$ is a uniformly scaling measure which generates an ergodic fractal distribution and $\mathcal{H}^1({\rm spt}\,\mu) > 0$, writing 
$$
D({\rm spt}\,\mu) = \lbrace |x-y|:\ x,y \in {\rm spt}\,\mu \rbrace
$$
for the distance set of the support of $\mu$, we have
\begin{equation}\label{distanceset}
\dim D( {\rm spt}\,\mu ) \geq \min \lbrace 1, \dim \mu \rbrace.
\end{equation}
This is \cite[Theorem 1.7]{FFS}. 

Another application of the uniform scaling property is found in the prevalence of normal numbers in the support of the measure. Recall that a number $x \in [0,1]$ is $a$-\emph{normal} if the sequence $\lbrace a^k x \mod 1 \rbrace_{k \in \N}$ equidistributes for Lebesgue measure, and a measure $\mu$ is called \emph{pointwise} $a$-\emph{normal} if $\mu$-almost every $x$ is $a$-normal. In a breakthrough paper \cite[Theorem 1.1]{HS_equidistribution}, Hochman and Shmerkin established a condition for a uniformly scaling measure to be pointwise normal in terms of dynamical properties of its tangent distribution: If $\mu$ is uniformly scaling, generates an ergodic fractal distribution $P$ and the pure-point spectrum of $P$ does not contain a non-zero integer multiple of $\frac{1}{\log a}$, where $a > 1$ is a Pisot number, then $\mu$ is pointwise $a$-normal.

Motivated in part by this progression, much attention has lately been given to the problem of understanding the scenery flow for various classes of fractal measures, and whether they are uniformly scaling and generate ergodic fractal distributions. Examples of uniformly scaling measures include ergodic measures in the one-sided shift space \cite{FP}, the occupation measure on Brownian motion in dimension at least $3$ \cite{Gavish}, self-affine measures on Bedford-McMullen-type carpets \cite{FFS} and self-similar measures with the \emph{open set condition} \cite{Gavish}. Recall that a measure $\mu$ on $\R^d$ is \emph{self-similar} if there exists a finite family of strictly contracting, angle-preserving affine maps $\lbrace \varphi_i \rbrace_{i=1}^m$ and a probability vector $(p_i)_{i=1}^m$ with positive entries such that 
\begin{equation}\label{selfsimilarintro}
\mu = \sum_{i=1}^m p_i \cdot \mu \circ \varphi_i^{-1}.
\end{equation}
For measures on $\R$ that satisfy \eqref{selfsimilarintro} for some equicontractive family $\lbrace \varphi_i \rbrace_{i=1}^m$, De-Jun Feng \cite{Fengpreprint} has shown that the open set condition can be relaxed to the strictly weaker \emph{finite type condition}. The separation conditions play crucial roles in \cite{Gavish} and \cite{Fengpreprint}, and this gives rise to the natural question: Are all self-similar measures uniformly scaling?

The main result of this paper, Theorem \ref{uniformly_scaling}, provides an affirmative answer to this question under a separation condition called the \emph{weak separation condition}. We say that the family $\lbrace \varphi_i \rbrace_{i=1}^m$ (or the associated self-similar measure) in \eqref{selfsimilarintro} satisfies the weak separation condition if the identity is not an accumulation point of the topological group generated by $F^{-1} F$, where $F$ is the semigroup with identity generated by $\lbrace \varphi_i \rbrace_{i=1}^m$. (In Section \ref{Preliminaries}, Definition \ref{weak_separation}, we adapt a different but equivalent characterization that is more explicitly reflected in the geometry of the self-similar measure.) We remark that this condition is strictly weaker than the conditions present in \cite{Gavish} and \cite{Fengpreprint}; see \cite{Zerner, LNfinitetype, HHR} for discussion on the relationship between the three conditions. 

\begin{theorem}\label{uniformly_scaling}
If $\mu$ is a self-similar measure satisfying the weak separation condition, then $\mu$ is uniformly scaling and generates an ergodic fractal distribution.
\end{theorem}

As an immediate application to the geometry of the measure, this allows us to deduce \eqref{marstrand} and \eqref{distanceset} for self-similar measures satisfying the weak separation condition. Combining Theorem \ref{uniformly_scaling} with the methods of Hochman-Shmerkin \cite{HS_equidistribution}, we also obtain an application to their pointwise normality; see Corollary \ref{equidistribution} in Section \ref{Discussion}.

The idea in the proof of Theorem \ref{uniformly_scaling} is to approximate the flow $(\mu_{x,t})_{t \geq 0}$ with another flow that behaves statistically in a similar manner as the orbit of a point in the underlying shift space. We construct this flow by first producing a fixed reference measure $\nu$ on $\R^d$ via the weak separation condition (in \eqref{nudefinition}). Then, drawing from certain recurrence properties in the underlying symbolic space, we may ensure that $\mu_{x, t}$ is ``close to'' $\nu$ for a large proportion of scales $t$, in a quite strong sense which also ensures that the magnifications of $\mu_{x,t}$ are close to those of $\nu$. This is made precise in Propositions \ref{nu_p} and \ref{distribution_asymptotic}. We may then use tools from ergodic theory to deduce that the approximating flow equidistributes for a limiting measure on the space of measures independent of the point $x$. Finally, a compactness argument allows us to deduce the convergence of the scenery flow of $\mu$. It then follows immediately from results of Hochman that the generated distribution is an ergodic fractal distribution.

In the presence of the open set condition, approximating $(\mu_{x,t})_{t \geq 0}$ with another flow is not necessary to gain access to tools from ergodic theory. Indeed, it follows from the open set condition and the relation \eqref{selfsimilarintro} that $\mu$ is homogeneous (in the sense of Gavish \cite{Gavish}), i.e. any measure in the (weak-$^*$) closure of the family $\lbrace \mu_{x,t}:\ x \in \R^d, t \geq 0 \rbrace$ equals a restriction of a scaled and translated copy of $\mu$ itself. Drawing from the CP-process machinery of Furstenberg, Gavish \cite{Gavish} showed that this closure always contains many uniformly scaling measures, in fact when $\mu$ is \emph{any} Radon measure on $\R^d$. Combining this with the homogeneity of $\mu$, Gavish concluded that self-similar measures with the open set condition are uniformly scaling. 

Another argument was presented by Hochman in \cite{Hochman} assuming the \emph{strong separation condition} under which the measures in the sum of \eqref{selfsimilarintro} have disjoint supports. In particular, if $K$ denotes the support of $\mu$, then the map $T: K \to K$, $x \mapsto \varphi_i^{-1}(x)$ where $1 \leq i \leq m$ is the unique number such that $x\in \varphi_i(K)$, is well-defined and preserves $\mu$. As described in \cite[Section 4]{Hochman}, the scenery flow of $\mu$ now arises as a factor of a suspension flow of the dynamical system $(K, \mu, T)$ (or of a certain skew-product dynamical system, when the linear parts of $\varphi_i$ are not positive scalars). Applying tools from ergodic theory on this suspension flow allowed Hochman to establish the uniform scaling property for self-similar measures satisfying the strong separation condition. 

Since a self-similar measure satisfying the weak separation condition is not necessarily homogeneous, nor is it in general preserved under any obvious map on $\R^d$, the existing arguments do not directly apply in our setting. On the other hand, our argument relies heavily on the weak separation condition in providing the reference measure $\nu$, which leaves open the question whether self-similar measures without this separation condition are uniformly scaling. 

The paper is organized as follows. In Section \ref{Preliminaries}, we set up our notation and recall the basic theory of dynamical systems and iterated function systems. In Section \ref{The scenery flow} we prove Theorem \ref{uniformly_scaling} in the simplified case when the defining IFS consists of only homotheties, i.e. functions composed of scaling and translation operations. Afterwards, we discuss the minor changes required to prove the general case. Section \ref{Asymptotics} is devoted to the proofs of Propositions \ref{nu_p} and \ref{distribution_asymptotic}, the main tools required in the proof of Theorem \ref{uniformly_scaling}. In Section \ref{Discussion} we briefly discuss the second assertion of Theorem \ref{uniformly_scaling}, a generalization of Theorem \ref{uniformly_scaling} for ergodic Markov measures, and its application to pointwise normality of self-similar measures.

\section{Preliminaries}\label{Preliminaries}

In this paper, a measure always refers to a Radon measure on a metrizable topological space. For a compact subset $A$ of a metric space $X$, write $\mathcal{P}(A)$ for the space of probability measures on $X$ for which ${\rm spt}\,\mu \subseteq A$. This space is endowed with the compact weak-$^*$ topology which we metrize with the Prokhorov metric $d$. We use the notation $\Vert \cdot \Vert$ to denote the total variation norm. It is well-known that convergence in the metric induced by the total variation norm implies convergence in the Prokhorov metric. Finally, the Euclidean norm is denoted by $|\cdot|$ and the supremum norm by $\Vert \cdot \Vert_\infty$. 

For a measure $\nu$ and all non-null measurable sets $A$, write $\nu|_A: B \mapsto \nu( B \cap A)$ for the restriction of $\nu$ to $A$ and $\nu_A = \nu(A)^{-1} \nu|_A$ for the normalized restriction. For a $\nu$-measurable function $f$, write $f\nu = \nu \circ f^{-1}$ for the image (or push-forward) measure, and if $f$ is integrable, write $f d\nu: A \mapsto \int \chi[A](x) f(x) \,d\nu(x)$ for the weighted measure, where $\chi[A]$ stands for the indicator function of $A$. 

\subsection{The scenery flow}

Following the notation of \cite{Hochman}, for $t \geq 0$ and $x \in \R^d$, we write $S_t:\ y \mapsto e^t y$ for the exponential scaling map and $T_x:\ y \mapsto y-x$ for the translation which takes $x$ to the origin. The function $S$ induces an additive action of $[0, +\infty)$ on $\mathcal{P}(B(0,1))$, given for every $t$ by $S_t^*: \nu \mapsto C (S_t \nu)|_{B(0,1)}$, where $C$ is such that $S_t^* \nu$ is a probability measure. Here and throughout, $B(0,1)$ denotes the closed unit ball. We sometimes refer to $S^*$ as the ``zoom-in'' operation and point out that in \cite{Hochman} it was denoted by $S^\square$ in order to emphasize restriction to $B(0,1)$. For the composed scaling and translation operation, we use the short-hand notation
$$
\nu_{x,t} := S_t^* T_x \nu.
$$

Using the introduced notation, we recall the definition of the scenery flow of a measure $\nu$ at a point $x \in {\rm spt}\,\nu$ as the flow
$$
\left( \frac{1}{T} \int_0^T \delta[\nu_{x, t}]\,dt \right)_{T > 0}.
$$
Note that this is a flow on $\mathcal{P}(\mathcal{P}(B(0,1)))$. The elements of $\mathcal{P}(\mathcal{P}(B(0,1)))$ are termed \emph{distributions} to emphasize that they are measures on the space of measures. The space of distributions is endowed with the weak-$^*$ topology. As a consequence of the compactness of this topology, the set of accumulation points of the scenery flow at any point is non-empty. Recall that a measure $\nu$ is termed uniformly scaling if there exists a distribution $P$ such that the scenery flow of $\nu$ converges to $P$ $\nu$-almost everywhere, as $T \to \infty$.

\subsection{The symbolic space}

Let $\Gamma = \lbrace 1, \ldots, m \rbrace$ be an alphabet with $m \geq 2$, and let $\Phi = \lbrace \varphi_i \rbrace_{i \in \Gamma}$ be a finite family of contractions on $\R^d$. The family $\Phi$ is called an iterated function system (IFS). It is well-known (see e.g. \cite{Falconerbook}) that there is a unique compact set $K$, called the \emph{attractor} of $\Phi$, such that
$$
K = \bigcup_{i \in \Gamma} \varphi_i(K).
$$
If the functions $\varphi_i$ are of the form $\varphi_i(x) = \rho_i R_i x + a_i$, where $0 < \rho_i < 1$, $R_i$ is an orthogonal matrix on $\R^d$ and $a_i \in \R^d$, the IFS $\Phi$ is called \emph{self-similar} and its attractor a self-similar set. Since our study does not depend on the coordinate basis of $\R^d$, in order to simplify notation we may and do suppose that $0 \in K$ and $K \subset B(0,1)$, where the inclusion is strict. 

Write $\Gamma^* = \bigcup_n \Gamma^n$ for the set of finite vectors, or \emph{words}, composed of elements of $\Gamma$. We usually denote finite words by characters $\ifi$, $\jfi$ and $\lfi$. For a finite word $\ifi = (i_0, i_1 ,\ldots, i_n) \in \Gamma^*$, write $\varphi_\ifi =\varphi_{i_0} \circ \cdots \circ \varphi_{i_n}$ and $K_\ifi = \varphi_\ifi(K)$. The notation $|\ifi|$ stands for the number of elements in the word $\ifi$ and is called the length of $\ifi$. For finite words $\ifi=(i_0, \ldots, i_n)$ and $\jfi=(j_0, \ldots, j_m)$, we let $\ifi \jfi:=(i_0, \ldots, i_n, j_0, \ldots, j_m) \in \Gamma^{|\ifi|+|\jfi|}$ denote their concatenation.

Elements of $\Gamma^\N$ are called infinite words and we often denote them by characters $\io$ and $\jo$. For a word $\io$, either infinite or of length $|\io| \geq k$, we write $\io|_k \in \Gamma^k$ for its projection to the first $k$ coordinates. Given $\ifi \in \Gamma^*$, we call elements of the set $\lbrace \jfi \in \Gamma^\N \cup \bigcup_{k \geq |\ifi|} \Gamma^k:\ \jfi|_{|\ifi|} = \ifi \rbrace$ \emph{descendants} of $\ifi$; the words $\ifi|_k$, $k = 1, 2, \ldots, |\ifi|$ are called the \emph{ancestors} of $\ifi$. These notions extend in the obvious way to the sets $K_{\ifi}$ and the functions $\varphi_{\ifi}$. For a finite word $\jfi \in \Gamma^*$, we write $[\jfi]$ for the cylinder set $\lbrace \io \in \Gamma^\N:\ \io|_{|\jfi|} = \jfi \rbrace$. We also use the term ``cylinder'' to refer to the sets $K_{\ifi}$. The collection of all cylinder sets generate the topology with which we equip $\Gamma^\N$.

Let $\sigma: (i_0, i_1 ,\ldots) \mapsto (i_1, i_2, \ldots)$ denote the continuous left-shift on $\Gamma^\N$. We say that a measure $\nu$ on $\Gamma^\N$ is \emph{invariant} with respect to $\sigma$ if $\sigma \nu = \nu$, and \emph{ergodic}, if $\nu(E) \in \lbrace 0,1 \rbrace$ for all sets $E$ satisfying $\sigma^{-1}E = E$. 

It is sometimes convenient to consider the \emph{two-sided extension} $(\Gamma^\Z, \sigma)$ of the dynamical system $(\Gamma^\N, \sigma)$, where $\sigma(\ldots, i_{-1}; i_0, i_1, \ldots) = (\ldots, i_{0}; i_1, i_2, \ldots)$ is also used to denote the left-shift on $\Gamma^\Z$ which, we note, is now invertible. For $\io = (\ldots, i_{-1}; i_0, i_1, \ldots) \in \Gamma^\Z$ and $n \leq m \in \Z$, write $\io|_{n}^m = (i_n, i_{n+1} \ldots, i_{m-1}, i_m) \in \Gamma^*$ and $[\io]_n^m = \lbrace (\ldots, j_{-1};j_0, j_1 ,\ldots) \in \Gamma^\Z:\ (j_n, j_{n+1}, \ldots ,j_{m-1}, j_m) = \io|_n^m \rbrace$. The topology of $\Gamma^\Z$ is again generated by the cylinder sets $[\io]_n^m$. There exists a natural bijection between the spaces of $\sigma$-invariant probability measures on $\Gamma^\N$ and $\Gamma^\Z$, given by $\nu \mapsto \nu^*$, where $\nu^*([\io]_n^m) = \nu([\io|_{n}^m])$ for every $\io \in \Gamma^\Z$ and $n \leq m$. It is straightforward to check that this bijection maps ergodic measures to ergodic ones. 

Write $\pi$ for the surjective map $\Gamma^\N \to K$, $$\io \mapsto \lim_{k \to \infty} \varphi_{\io|_k}(0),$$
which we call the natural projection. Slightly abusing notation, we often write $T_{\io} := T_{\pi(\io)}$ for the translation taking $\pi(\io)$ to the origin.

Fix a probability vector $p = (p_i)_{i \in \Gamma}$ and write $\bar{\mu} = p^\N$ for the associated Bernoulli measure on $\Gamma^\N$. It is well-known that the measure $\pi\bar{\mu} =: \mu$ is a Radon measure supported on the self-similar set $K$ and satisfies 
$$
\mu = \sum_{i \in \Gamma} p_i \cdot \varphi_i \mu.
$$
This measure is called the \emph{self-similar measure} associated to the vector $p$. Without loss of generality, we assume that $p$ has strictly positive entries: Those $i$ with $p_i = 0$ can be removed from $\Gamma$ without affecting $\mu$, and if only one entry is positive, $\mu$ is a point mass and trivially uniformly scaling.

\subsection{The weak separation condition}\label{WSC}

To make the structure of self-similar sets and measures more tractable, there are numerous separation conditions one can impose on the defining IFS. These conditions are meant to limit the ways the cylinders $K_{\ifi}$ can overlap with one another, making it easier to understand the geometric properties of self-similar measures by studying the dynamical system $(\Gamma^\N, \mu, \sigma)$. Understanding the geometric properties of self-similar sets and measures without assuming any separation condition is an interesting and challenging problem in fractal geometry; see \cite{Hochmanoverlaps, Shmerkinlq, Varju, HochmanICM} and the references therein for some recent breakthroughs on the topic.

The \emph{weak separation condition} was first introduced by Lau and Ngai in \cite[Definition 6.2]{LN}, where the authors studied the multifractal formalism for self-similar measures satisfying the condition. It allows overlaps between the cylinders $K_{\ifi}$ but in a certain limited manner: Roughly speaking, for any $r>0$, any closed ball of radius $r$ may intersect at most $M$ distinct cylinders of diameter $\approx r$, for some $M$ independent of $r$. We point out that no bound is imposed on the number of \emph{exact overlaps}: Given $\ifi \in \Gamma^*$, there might be numerous words $\jfi \in \Gamma^*$ for which $K_{\ifi} = K_{\jfi}$. Before we formulate the separation condition precisely, some more notation is in order.

Let $\Phi = \lbrace \varphi_i(x) = \rho_i R_i x + a_i \rbrace_{i \in \Gamma}$ be a self-similar IFS. For a finite word $\ifi = (i_0, \ldots, i_{n}) \in \Gamma^*$, write $\ifi^- = (i_0, \ldots, i_{n-1})$, $\rho_{\ifi} = \rho_{i_0} \cdots \rho_{i_n}$, $B_{\ifi} = \varphi_{\ifi}(B(0,1))$, and denote
$$
\cN(\ifi) = \lbrace \varphi_{\ifi}^{-1} \circ \varphi_{\jfi}:\ \jfi \in \Gamma^*,\ \rho_{\jfi} \leq \rho_{\ifi} < \rho_{\jfi^-},\ K_{\jfi} \cap B_{\ifi} \neq \emptyset \rbrace.
$$
We call the collection $\cN(\ifi)$ the \emph{neighbourhood system} of $\ifi$. The sets $\varphi_{\ifi} \circ f(K)$, $f \in \cN(\ifi)$, are called the \emph{neighbours} of $K_{\ifi}$.

\begin{definition}\label{weak_separation}
The IFS $\Phi$ satisfies the \emph{weak separation condition}, if
\begin{equation}\label{WSCequation}
\sup_{\ifi \in \Gamma^*} \# \cN(\ifi) < \infty.
\end{equation}
\end{definition}

\begin{remark}\label{remarkwsc}
If $K$ is not contained in an affine hyperplane and we write $\cN'(\ifi) = \lbrace f \in \cN(\ifi):\ f(K) \cap {\rm int}\, \CC K \neq \emptyset \rbrace$, where ${\rm int}$ denotes the interior and $\CC$ denotes the convex hull, then it is not difficult to see that \eqref{WSCequation} is equivalent to the condition $\sup_{\ifi \in \Gamma^*} \# \cN'(\ifi) < \infty$. In fact\footnote{We thank Alex Rutar for suggesting this remark.}, when ${\rm int}\, \CC K \neq \emptyset$, it is an interesting question whether \eqref{WSCequation} is equivalent to the a-priori stronger condition
\begin{equation}\label{equivalentcondition}
\# \left( \bigcup_{\ifi \in \Gamma^*} \cN'(\ifi)\right) < \infty.
\end{equation}
This question was first raised by Lau and Ngai in \cite{LNfinitetype}. Some evidence for the equivalence was provided in \cite{HHR} where the authors showed that \eqref{WSCequation} implies \eqref{equivalentcondition} when $K = [0,1]$. 
\end{remark}

While appearing different in form, Definition \ref{weak_separation} coincides with that of Lau and Ngai; this is proved in Zerner's paper \cite[Theorem 1]{Zerner}, where multiple equivalent formulations (including the one in the introduction) for the weak separation condition are listed. We refer the reader to \cite{LW, FL, KR, FHOR} for further results on the structure of weakly separated self-similar sets and measures.

In our study of the scenery flow, we aim to establish regularity in $(\mu_{\io, t})_{t \geq 0}$ through the existence of a recurring neighbourhood system in the sequence $(\cN(\io|_k))_{k \in \N}$. To study this recurrence with the help of the theory of dynamical systems, we let $\ifi_0 \in \Gamma^*$ be a word for which
\begin{equation}\label{i_0}
\#\cN(\ifi_0) = \sup_{\ifi \in \Gamma^*} \# \cN(\ifi)
\end{equation}
and call the \emph{maximal neighbourhood (system)} the collection
\begin{equation}\label{maximalneighbourhood}
\cE := \cN(\ifi_0).
\end{equation}
The following observation regarding the recurrence of $\cE$ in the zoom-in process was made by Feng and Lau in \cite{FL}. It is one of the key tools in our proof. 

\begin{lemma}\label{maximal}
For any $\lfi \in \Gamma^*$, 
$$
\lbrace \varphi_\jfi:\ \jfi \in \Gamma^*,\ \rho_{\jfi} \leq \rho_{\kfi \ifi_0} \leq \rho_{\jfi^-},\ K_{\jfi} \cap B_{\lfi \ifi_0} \neq \emptyset \rbrace = \varphi_{\lfi \ifi_0} \cE,
$$
where we write $\varphi_{\lfi \ifi_0} \cE = \lbrace \varphi_{\lfi \ifi_0} \circ f:\ f \in \cE \rbrace$.
\end{lemma}

\begin{figure}[H]

\begin{tikzpicture}[scale=0.8]

\draw (0,0) rectangle (4,4) node[right] {$K$};
\draw[fill, lightgray]  (1.3, 1.7) node[left, black] {$K_{\kfi_1 \ifi_0}$} rectangle (2.6, 3);
%\draw[fill] (2.3, 2.5)  circle (0.05);
\draw  (1.3, 1.7) rectangle (2.6, 3);

\draw[fill, lightgray] (2.8, 0.3) rectangle (3.3, 0.8) node[above, black] {$K_{\kfi_2 \ifi_0}$};
\draw (2.8, 0.3) rectangle (3.3, 0.8);
%\draw[fill] (3, 0.5)  circle (0.05);
\draw[dotted, thin] (3.05, 0.55) circle (0.4);

\draw[dotted, thin] (1.95, 2.4) circle (1);
\draw[dotted, thin] (1.95, 3.4) -- (8.2, 4.5);
\draw[dotted, thin] (1.95, 1.4) -- (8.2, -0.5);
\draw[dotted, thin] (3.05, 0.95) -- (8.2, 4.5);
\draw[dotted, thin] (3.05, 0.15) -- (8.2, -0.5);
\draw[dotted, thin] (8.2, 2) circle (2.5);

\draw[fill, lightgray] (7, 1) rectangle (9, 3);
\draw(7, 1)  rectangle (9, 3);
\draw (6.5, 0.5) rectangle (8.5, 2.5);
\draw (6.7, 1.9) rectangle (8.7, 3.9);
\draw (7.7, 1.6) rectangle (9.7, 3.6);
\draw (8, 0.3) node[below, black] {$\lbrace f(K):\ f \in \cE \rbrace$} rectangle (10, 2.3);

\end{tikzpicture}
\caption{Here $\kfi_1, \kfi_2 \in \Gamma^*$. The neighbourhood system of any cylinder whose symbolic coding ends with $\ifi_0$ is $\cE$.}

\end{figure}

\begin{proof}
We have 
\begin{align*}
&\lbrace \varphi_{\jfi}:\ \jfi \in \Gamma^*,\ \rho_{\jfi} \leq \rho_{\kfi \ifi_0} \leq \rho_{\jfi^-},\ K_{\jfi} \cap B_{\lfi \ifi_0} \neq \emptyset \rbrace \\
\supset\ &\lbrace \varphi_{\lfi \ifi}:\ \ifi \in \Gamma^*,\ \rho_{\kfi \ifi} \leq \rho_{\kfi \ifi_0} \leq \rho_{\kfi \ifi^-},\ K_{\lfi \ifi} \cap B_{\lfi \ifi_0} \neq \emptyset \rbrace\\
=\ &\lbrace \varphi_{\lfi} \circ \varphi_{\ifi_0} \circ \varphi_{\ifi_0}^{-1} \circ \varphi_{\ifi}:\ \ifi \in \Gamma^*,\ \rho_{\ifi} \leq \rho_{ \ifi_0} \leq \rho_{\ifi^-},\ K_{ \ifi} \cap B_{\ifi_0} \neq \emptyset \rbrace \\
=\ &\varphi_{\lfi \ifi_0} \cE.
\end{align*}
This is equivalent to
\begin{equation}\label{included}
\cE \subseteq \lbrace \varphi_{\lfi \ifi_0}^{-1} \circ \varphi_{\jfi}:\ \jfi \in \Gamma^*,\ \rho_{\jfi} \leq \rho_{\kfi \ifi_0} \leq \rho_{\jfi^-},\ K_{\jfi} \cap B_{\lfi \ifi_0} \neq \emptyset \rbrace.
\end{equation}
Recall from \eqref{i_0} and \eqref{maximalneighbourhood} that $\cE$ has cardinality at least that of the right-hand side of \eqref{included}. Hence \eqref{included} holds with equality.
\end{proof}

\section{The scenery flow}\label{The scenery flow}

For simplicity, we first prove the following special case of Theorem \ref{uniformly_scaling}, and afterwards discuss the minor modifications required in the proof in order to lift these restrictions on the IFS.
\begin{theorem}\label{specialcase}
If $\Phi = \lbrace \phi_i(x) = \rho x + a_i\rbrace_{i \in \Gamma}$ is an IFS of equicontractive homotheties on $\R^d$ satisfying the weak separation condition, then any self-similar measure $\mu$ associated to $\Phi$ is uniformly scaling.
\end{theorem}
In the setting of Theorem \ref{specialcase}, the neighbourhood system of a finite word $\ifi \in \Gamma^*$ assumes the slightly simplified form
$$
\cN(\ifi) = \lbrace \varphi_{\ifi}^{-1} \circ \varphi_{\jfi}:\ \jfi \in \Gamma^{|\ifi|},\ K_{\jfi} \cap B_{\ifi} \neq \emptyset \rbrace.
$$

If $(r_n)_{n \in \N} \subseteq [0, +\infty)$ is a sequence such that 
\begin{equation}\label{conditionforrn}
\lim_{n \to \infty} \frac{r_n}{n} \in (0,+\infty),
\end{equation}
then the sequence
\begin{equation}\label{beginning}
\left( \frac{1}{r_n} \int_0^{r_n} \delta[\mu_{\io, t}] \,dt \right)_{n \in \N} 
\end{equation}
is asymptotic to the scenery flow of $\mu$ at $\pi(\io)$. Our strategy is for every $\varepsilon > 0$ and almost every $\io \in \Gamma^\N$ to choose the numbers $r_n = r_n(\io)$ in such a way that the tail of \eqref{beginning} is within Prokhorov distance $\varepsilon$ from a distribution $P_\varepsilon$ independent of $\io$. Provided that $(r_n)_{n \in \N}$ also satisfies \eqref{conditionforrn}, it follows that tail of the scenery flow is within distance $2\varepsilon$ from $P_\varepsilon$. Taking $\varepsilon \to 0$ along a countable sequence and using the compactness of the space of distributions, we then obtain a set of full $\mu$-measure in which the scenery flow converges. 

To find a distribution $P_\varepsilon$ which nearly captures the statistics of the scenery of $\mu$, we want to approximate $(\mu_{\io, t})_{t \geq 0}$ with a flow whose dynamics can be more easily traced back to $(\Gamma^\N, \bar{\mu}, \sigma)$. Since our prior hope of accomplishing this is to use the recurrence of $\cE$ in $(\cN(\io|_k))_{k \in \N}$ which follows from the recurrent visits of $(\sigma^k \io)_{k \in \N}$ in $[\ifi_0]$, our first task is to exhibit the scenery measures using the neighbourhood systems appearing around $\pi(\io)$. The following representation is immediate from the self-similarity of $\mu$; recall that $T_{\io}$ denotes the translation $x \mapsto x - \pi(\io)$. 

\begin{lemma}\label{selfsimilarity}
Let $\io \in \Gamma^\N$. For every $k \in \N$, there exists a probability vector $(q_f(\io|_k))_{f \in \cN(\io|_k)}$ such that
$$
\mu_{\io, t} = S_{t+ k \log \rho}^* T_{\sigma^k \io} \sum_{f \in \cN(\io|_k)} q_f(\io|_k) \cdot f \mu
$$
whenever $t > 0$ is such that $B(\pi(\io), e^{-t}) \subseteq B_{\io|_k}$.
\end{lemma}

\begin{proof}
An elementary calculation shows that 
\begin{equation}\label{elementary}
S_{k \log \rho^{-1}} T_{\io} = T_{\sigma^k \io} \varphi_{\io|_k}^{-1}.
\end{equation}
Since we suppose that $B(\pi(\io), e^{-t}) \subseteq B_{\io|_k}$, using self-similarity, \eqref{elementary} and the additivity of the action induced by $S^*$ we can write
\begin{align*}
\mu_{\io, t} &= S_{t+k\log \rho}^* S_{k \log \rho^{-1}} T_{\io} \sum_{\ifi \in \Gamma^k,\ K_{\ifi} \cap B_{\io|_k} \neq \emptyset} p_{\ifi} \cdot \varphi_{\ifi} \mu\\
&= S^*_{t+ k\log\rho} T_{\sigma^k \io} \sum_{\ifi \in \Gamma^k,\ K_{\ifi} \cap B_{\io|_k} \neq \emptyset} p_{\ifi} \cdot (\varphi_{\io|_k}^{-1} \circ \varphi_{\ifi}) \mu
\end{align*}
and using the definition of $\cN(\io|_k)$, the sum in this representation can be reordered as
$$
\sum_{f \in \cN(\io|_k)} \left( \sum_{\ifi \in \Gamma^k,\ \varphi_{\io|_k}^{-1} \circ\varphi_{\ifi} =  f} p_{\ifi} \right) \cdot f \mu.
$$
Normalizing gives the statement.
\end{proof}

Motivated by the recurrence of $\cE$ in $(\cN(\io|_k))_{k \in \N}$ and the prescribed representation for the scenery measures, we define the measure
\begin{equation}\label{nudefinition}
\nu = \sum_{f \in \cE} f \mu.
\end{equation}
This is the recurring ``frame'' to which we compare the scenery measures. Indeed, for those $k$ for which $\cN(\io|_k) = \cE$, Lemma \ref{selfsimilarity} tells us that $\mu_{\io, t} \ll \nu_{\sigma^k \io, t+k\log \rho}$ with density $\zeta = \zeta(\io|_k)$ determined by the weights $q_f(\io|_k)$. While the exact form of the density depends on the entirety of $\io|_k$, the content of the following proposition is that much less information on $\io|_k$ is required to control $\zeta(\io|_k)$ on a sufficient level.

\begin{proposition}\label{nu_p}
There exists a finite word $\jfi_0 \in \Gamma^*$ and a family of $\nu$-integrable functions $\lbrace \zeta_h:\ h \in \cF \subseteq \cE\rbrace$, each bounded from below by a constant $C_h > 0$, such that for almost every $\io \in \Gamma^\N$ and all $k \in \N$ for which $\sigma^{k-|\ifi_0 \jfi_0|} \io \in [\ifi_0 \jfi_0]$, there exists a function $\zeta = \zeta(\io|_k)$ which is a convex combination of the functions $\zeta_h$ such that 
$$
\mu_{\io, t} = \left( \zeta d\nu\right)_{\sigma^{k} \io, t + k \log \rho}
$$
for any $t > 0$ satisfying $B(\pi(\io), e^{-t}) \subseteq B_{\io|_{k}}$.

\end{proposition}
The proof requires some geometric work and can be found in Section \ref{Asymptotics}. 

We write $$\cW = \lbrace \zeta(\io|_k):\ \io\ {\rm and}\ k\ \text{satisfy the assumptions of Proposition \ref{nu_p}} \rbrace$$ for the countable collection of all functions given by Proposition \ref{nu_p}. Note that as a bounded subset of the linear span of a finite family of functions, the closure of $\cW$ is compact in the topology induced by the supremum norm. Combining this compactness with an application of the Lebesgue-Besicovitch differentiation theorem and some elementary continuity properties of $\zeta \mapsto (\zeta d\nu)_{\io, t}$ we can deduce the following:

\begin{proposition}\label{distribution_asymptotic}
For all $\varepsilon > 0$, there exists a set $\cJ = \cJ(\varepsilon) \subseteq \Gamma^\N$ and an integer $N = N(\varepsilon)$ such that $\bar{\mu}(\cJ) > 0$ and for any $\zeta \in \cW$ and every $\io \in \cJ$, 
\begin{equation}\label{muarbitrary}
d \left( \frac{1}{T} \int_0^{T} \delta[\mu_{\io,t}]\,dt,\ \frac{1}{T} \int_0^{T}  \delta[(\zeta d\nu)_{\io,t}] \,dt \right) < \varepsilon
\end{equation}
for all $T \geq N$.
\end{proposition}

The proof can be found in Section \ref{Asymptotics}. 

These propositions together give us a way to approximate the statistics of $(\mu_{\io, t})_{t \geq 0}$ using only the statistics of the recurrence of $(\sigma^k \io)_{k \in \N}$ in $[\ifi_0 \jfi_0]$.

\subsection{The sequence $(r_n)_n$}

Let $\varepsilon > 0$. We now move on to defining the distribution of the random sequence $(r_n)_{n \in \N} $; recall \eqref{beginning} and the discussion thereafter. Let $\jfi_0$ be the finite word given by Proposition \ref{nu_p}, and let $\cJ = \cJ(\varepsilon)$ be the set given by Proposition \ref{distribution_asymptotic}. Since everything we define in the following depends on the set $\cJ$ and hence on $\varepsilon$, we suppress this dependence from our notation.

We will first define, for almost every $\io \in \Gamma^\N$, a sequence $(t_n)_{n \in \N} \subseteq \N$ such that for every $n$,
\begin{itemize}
\item[i)] $\sigma^{t_n} \io \in \cJ$ and
\item[ii)] $\sigma^{t_n - |\ifi_0 \jfi_0|} \io \in [\ifi_0 \jfi_0]$
\end{itemize} 
In particular, the sequence $(\sigma^{t_n} \io)_n$ must also include some information on the ``past'' of each element. One way to deal with this is to consider the two-sided extension of the space $\Gamma^\N$. Since there is a natural bijection between the spaces of invariant measures on $\Gamma^\N$ and $\Gamma^\Z$, it is convenient to denote the extended Bernoulli measure on $\Gamma^\Z$ also by $\bar{\mu}$. Since the natural projection $\pi$ extends to $\Gamma^\Z$ by $\io \mapsto \pi(\io^+)$, where $(\cdot)^+$ denotes the projection to the positive coordinates, and both the Bernoulli measure and the extended one project onto the self-similar measure $\mu$, we do not usually make a distinction between one-sided and two-sided infinite words. 

We let
$$
\cJ' = \lbrace \io \in \Gamma^\Z:\ \io|_{-|\ifi_0 \jfi_0|}^{-1} = \ifi_0 \jfi_0,\ \io^+ \in \cJ \rbrace
$$ 
and define the sequence $(t_n)_{n \in \N}$ as the return times of $\io$ to the set $\cJ'$: write 
$$
\tau:\ \jo \mapsto \min \lbrace n \geq 1:\ \sigma^{n} \jo \in \cJ'\rbrace
$$
for the time of first visit in $\cJ'$, defined $\bar{\mu}$-almost everywhere, and set
\begin{align*}
t_0 &= 0, \\
t_n(\io) &= t_{n-1}(\io) + \tau(\sigma^{t_{n-1}} \io)\ \text{for}\ n \geq 1.
\end{align*}
Note that $\bar{\mu}(\cJ') > 0$. We also point out the following:
\begin{lemma}\label{tauintegrable}
The function $\tau$ is $\bar{\mu}$-integrable over $\cJ'$.
\end{lemma}
\begin{proof}
Since $\bar{\mu}$ is ergodic, this is implied by Kac's theorem \cite[Theorem 2]{Kac}.
\end{proof}

Finally, define 
\begin{align*}
r_n(\io) &= \inf \lbrace r\geq 0:\ B(\pi(\io), e^{-r}) \subseteq B_{\io|_{t_{n}(\io)} } \rbrace \\
&= t_n(\io) \log \rho^{-1} + \alpha_0(\sigma^{t_n(\io)} \io).
\end{align*}
Here
\begin{align*}
\alpha_0:\ \jo \mapsto \inf \lbrace r \geq 0:\ B(\pi(\jo), e^{-r}) \subseteq B(0,1)\rbrace
\end{align*} 
is the correction that arises because the natural projections of $\jo \in \Gamma^\N$ are not necessarily in the ``center'' of $B(0,1)$. Since we suppose that $K$ is strictly included in $B(0,1)$, for a large enough integer $M$ we have $\alpha_0(\jo) \in [0, M]$ for all $\jo\in\Gamma^\N$. 

Let $N$ be the integer given by Proposition \ref{distribution_asymptotic}. Note that since we have
\begin{align*}
r_{n+1}(\io)-r_n(\io) &= (t_{n+1}(\io) - t_{n}(\io)) \log \rho^{-1} + \alpha_0(\sigma^{t_{n+1}} \io) - \alpha_0(\sigma^{t_n} \io) \\
&\geq \log \rho^{-1} - M,
\end{align*}
by replacing the alphabet $\Gamma$ with a high-level iteration $\Gamma^{\lceil (M+N)/\log \rho^{-1}\rceil}$ we can suppose that the difference $r_{n+1}(\io) - r_n(\io) \geq N$ for almost every $\io \in \cJ'$. We will later see that this sequence satisfies the condition \eqref{conditionforrn}. 

We are now ready to prove Theorem \ref{specialcase}.

\subsection{Proof of Theorem \ref{specialcase}}\label{ergodicpart}

In the proof, we require the following application of Birkhoff's ergodic theorem. We postpone its short proof to the end of the section.
\begin{claim}\label{asymptotic}
For any $\bar{\mu}$-integrable (distribution- or real-valued) $g$ and almost every $\io$,
$$
\lim_{n \to \infty} \frac{1}{n} \sum_{k=0}^{n-1} g(\sigma^{t_k} \io) = \int g \,d\bar{\mu}_{\cJ'}.
$$
\end{claim}
Define the functions $\psi, \eta$ for $\bar{\mu}$-almost every $\jo$ by
\begin{align}\label{psi_r}
\begin{split}
&\psi:\ \jo \mapsto \int_{\alpha_0(\jo)}^{\tau(\jo) \log \rho^{-1} + \alpha_0(\sigma^{\tau(\jo)}\jo)} \delta[\mu_{\jo, t}]\,dt, \\
&\eta:\ \jo \mapsto \tau(\jo)\log \rho^{-1} + \alpha_0(\sigma^{\tau(\jo)}\jo) - \alpha_0(\jo).
\end{split}
\end{align}
The measurability of these functions is standard, and integrability over $\cJ'$ follows from that of $\tau$ (Lemma \ref{tauintegrable}) and the assumption $\alpha_0(\cJ) \subseteq [0,M]$. 

\begin{proof}[Proof of Theorem \ref{specialcase}]
Let now $\io\in \bigcup_{n=0}^\infty \sigma^{-n} \cJ'$ be such that Claim \ref{asymptotic} holds for the functions $\psi$ and $\eta$. The set of these words has full $\bar{\mu}$-measure. Since the scale we start the zoom-in process at has no effect on the generated distribution, we may suppose that $\io \in \cJ'$. To make the notation more clear, from now on we suppress the dependance of $r_k$ and $t_k$ on $\io$. Recall that the set $\cJ'$ was defined so that whenever $\io \in \cJ'$, we have $\io^+ \in \cJ$ and $(\sigma^{-|\ifi_0 \jfi_0|} \io)^+ \in [\ifi_0 \jfi_0]$. 

Now, we can write the sequence \eqref{beginning} as
\begin{align}\label{givesmass}
\frac{1}{r_n} \int_{0}^{r_n} \delta[\mu_{\io,t}]\,dt &= \frac{1}{r_n} \sum_{k=0}^{n-1} \int_{r_{k}}^{r_{k+1}} \delta[\mu_{\io,t}]\,dt\nonumber\\
&= \frac{1}{r_n}\sum_{k=0}^{n-1} \int_{t_k \log \rho^{-1} + \alpha_0(\sigma^{t_k}\io)}^{t_{k+1} \log \rho^{-1}+ \alpha_0 (\sigma^{t_{k+1}} \io)} \delta[(\zeta(\io|_k)d\nu)_{\sigma^{t_k} \io, t+t_k\log \rho}]\,dt\nonumber\\
&= \frac{1}{r_n} \sum_{k=0}^{n-1}\int_{\alpha_0(\sigma^{t_k} \io)}^{\tau(\sigma^{t_k} \io) \log \rho^{-1} + \alpha_0(\sigma^{t_{k+1}} \io)} \delta[(\zeta(\io|_k)d\nu)_{\sigma^{t_k} \io, t}]\,dt,
\end{align}
using Proposition \ref{nu_p} and writing out the definitions of $r_k$ in the second equality, and performing a change of variable $t \mapsto t - t_k \log \rho$ in the last one. Combining this with Proposition \ref{distribution_asymptotic} we get that
\begin{equation}\label{muestimate}
d \left( \frac{1}{r_n} \int_{0}^{r_n} \delta[\mu_{\io,t}]\,dt,\ \frac{1}{r_n} \sum_{k=0}^{n-1}\int_{\alpha_0(\sigma^{t_k} \io)}^{\tau(\sigma^{t_k} \io) \log \rho^{-1} + \alpha_0(\sigma^{t_{k+1}} \io)} \delta[\mu_{\sigma^{t_k} \io, t}]\,dt\right)  < \varepsilon
\end{equation}
for all large enough $n$. Write
$$
Q_n := \frac{1}{r_n} \sum_{k=0}^{n-1}\int_{\alpha_0(\sigma^{t_k} \io)}^{\tau(\sigma^{t_k} \io) \log \rho^{-1} + \alpha_0(\sigma^{t_{k+1}} \io)} \delta[\mu_{\sigma^{t_k} \io, t}]\,dt.
$$

Using the functions $\psi$ and $\eta$ defined in \eqref{psi_r}, we have for $Q_n$ the representation
\begin{align*}
Q_n &= \frac{1}{\sum_{k=0}^{n-1} (r_{k+1} - r_k)} \sum_{k=0}^{n-1} \psi(\sigma^{t_k} \io) 
= \left(\frac{1}{\frac{1}{n}\sum_{k=0}^{n-1} \eta(\sigma^{t_k} \io)}\right) \left(\frac{1}{n}\sum_{k=0}^{n-1} \psi(\sigma^{t_k} \io)\right).
\end{align*}
It follows from Claim \ref{asymptotic} that the sequence $(Q_n)_{n \in \N}$ converges to a limiting distribution $P_\varepsilon$ as $n\to \infty$. Combining this with \eqref{muestimate}, it follows that there exists an integer $n_\varepsilon$ such that
\begin{equation*}
d \left( \frac{1}{r_n} \int_{0}^{r_n} \delta[\mu_{\io,t}]\,dt,\ P_\varepsilon\right) < 2\varepsilon\ \text{for all}\ n \geq n_\varepsilon.
\end{equation*}
Since Claim \ref{asymptotic} applied for $\eta$ shows that 
$$
\lim_{n \to \infty} \frac{r_n}{n} = \lim_{n \to \infty} \frac{1}{n} \sum_{k=0}^{n-1} \eta(\sigma^{t_k} \io) = \int \eta \,d\bar{\mu}_{\cJ'} \in (0, +\infty),
$$
it follows that the sequence
$$
\left( \frac{1}{r_n} \int_{0}^{r_n} \delta[\mu_{\io,t}]\,dt \right)_{n \in \N}
$$ 
is asymptotic to the scenery flow of $\mu$ at $\pi(\io)$. 

We have now found, for every $\varepsilon > 0$, a distribution $P_{\varepsilon}$ and an integer $T_\varepsilon$ such that for almost every $\io$, 
\begin{equation}\label{conclusion}
d \left( \frac{1}{T} \int_0^T \delta[\mu_{\io, t}] \,dt, P_{\varepsilon} \right) < 2\varepsilon\ \text{for all}\ T \geq T_\varepsilon.
\end{equation}
Letting $\varepsilon \to 0$ along a countable sequence and taking the intersection of the sets of full $\bar{\mu}$-measure in which \eqref{conclusion} holds, we obtain a set of full measure in which \eqref{conclusion} holds for every $\varepsilon$. Using the compactness of $\mathcal{P}(\mathcal{P}(B(0,1)))$, it follows that the scenery flow of $\mu$ converges almost everywhere. Moreover, since the distributions $P_\varepsilon$ are independent of $\io$, so is the limit of the scenery flow. This concludes the proof of Theorem \ref{specialcase}.
\end{proof}

\begin{proof}[Proof of Claim \ref{asymptotic}]
Note first that the claim is equivalent to
\begin{align}\label{equivalentclaim}
\lim_{n \to \infty} \frac{1}{\bar{\mu}(\cJ')n} \sum_{k=0}^{\lfloor \bar{\mu}(\cJ')n \rfloor} g(\sigma^{t_k} \io) = \int g \,d\bar{\mu}_{\cJ'}.
\end{align}
Using Birkhoff's ergodic theorem for $\sigma$ and $\chi[\cJ']$ in the form
$$
\lim_{n \to \infty} \frac{\#\lbrace 0 \leq \ell \leq n-1:\ \sigma^\ell \io \in \cJ' \rbrace}{n} = \bar{\mu}(\cJ'),
$$ 
we see that the sequence on the left-hand side of \eqref{equivalentclaim} is asymptotic to the sequence
$$
\frac{1}{\bar{\mu}(\cJ')n} \sum_{k=0}^{\#\lbrace 0 \leq \ell \leq n-1:\ \sigma^\ell \io \in \cJ' \rbrace}g(\sigma^{t_k} \io).
$$
Now, using the fact that $\sigma^\ell \io \in \cJ'$ if and only if $\ell = t_k$ for some $k$, which follows from the definition of $(t_k)_k$ as the return times to the set $\cJ'$, we can write
$$
\frac{1}{\bar{\mu}(\cJ')n} \sum_{k=0}^{\#\lbrace 0 \leq \ell \leq n-1:\ \sigma^\ell \io \in \cJ' \rbrace}g(\sigma^{t_k} \io) = \frac{1}{\bar{\mu}(\cJ')}\frac{1}{n} \sum_{\ell=0}^{n-1} \chi[\cJ'](\sigma^\ell \io) g(\sigma^\ell \io).
$$
Finally, using Birkhoff's ergodic theorem for $\chi[\cJ'] g$ and $\sigma$, we obtain \eqref{equivalentclaim}.
\end{proof}

\subsection{On the proof of Theorem \ref{uniformly_scaling}}\label{subsection} In the proof of Theorem \ref{specialcase} it was assumed that the IFS $\Phi$ consists of equicontractive homotheties. Only minor modifications in the proof are required in order to dispose of this requirement. We provide an overview of the needed modifications; details are left to the interested reader. For now, suppose that $\Psi = \lbrace \varphi_i(x) = \rho_i R_i x + a_i \rbrace_{i \in \Gamma}$ is a self-similar IFS satisfying the weak separation condition and $\mu$ is an associated self-similar measure.

In Proposition \ref{nu_p}, to obtain the stated identity we need to replace the measure $(\zeta d\nu)_{\sigma^k, t+k\log \rho}$ by $(R_{\io|_k}(\zeta d \nu))_{\sigma^k \io, t + \log \rho_{\io|_k}}$, where $R_{\io|_k}$ is the matrix component of $\varphi_{\io|_k}$. The statement of Proposition \ref{distribution_asymptotic} extends as it is for distributions generated by $R\nu$ and $R (\zeta d\nu)$, for any orthogonal matrix $R$. 

In the ergodic-theoretic part of Subsection \ref{ergodicpart}, one must take into account the additional dynamics coming from the matrix components of the similarities. Instead of the system $(\Gamma^\N, \bar{\mu}, \sigma)$, we need to consider the skew-product $(\Gamma^\N \times G, \bar{\mu} \times \gamma, F)$, where $G$ is the closure of the topological group generated by the matrix components of $\varphi_i$, $\gamma$ is the right-invariant Haar measure on $G$ and $F$ is the transformation $(\io, R) \mapsto (\sigma \io, R \cdot R_{\io|_1})$. The ergodicity of this kind of skew-product system is due to Kakutani \cite{Kakutani}. Going through Section \ref{The scenery flow} with these modifications, one sees that for $\gamma$-almost every $R$, the measure $R \mu$ is uniformly scaling. This property is clearly preserved under action by $R^{-1}$. 

The fact that $\mu$ generates an ergodic fractal distribution is immediate from the results of Hochman \cite{Hochman}. Since some additional terminology needs to be introduced before the short deduction, we postpone it to Section \ref{Discussion} in order not to disrupt the main line of argument behind the uniform scaling property.

\section{The scenery of the self-similar measure}\label{Asymptotics}

In this section, we prove Propositions \ref{nu_p} and \ref{distribution_asymptotic}. Although the proofs are written in the setting of Theorem \ref{specialcase} for the sake of notational simplicity, after modifying the statements of the propositions according to the highlights in Subsection \ref{subsection}, the arguments extend with little work to the setting of Theorem \ref{uniformly_scaling}. 

We adopt the notation
\begin{equation}\label{symbolicmag}
\mu_{\ifi} := (\varphi_{\ifi}^{-1} \mu)_{B(0,1)}
\end{equation}
for every $\ifi \in \Gamma^*$, for the ``symbolic magnifications'' of the self-similar measure.

\subsection{Proof of Proposition \ref{nu_p}}

Recall Lemma \ref{selfsimilarity} and that our motivation is to find a way to describe the probability vector $(q_f(\io|_k))_{f \in \cN(\io|_k)}$ without requiring information on the entirety of the word $\io|_k$. Recall also the choice of the word $\ifi_0$ from \eqref{i_0} and the definition of the maximal neighbourhood $\cE$ from \eqref{maximalneighbourhood}. Through the work in this subsection, we will find a finite word $\jfi_0$ such that for any $\kfi \in \Gamma^*$, $\cN(\kfi \ifi_0 \jfi_0) = \cE$ and the weights $q_f(\kfi \ifi_0 \jfi_0)$, $f \in \cE$, are bounded from below by a constant independent of $\kfi$.

The key geometric observation in the proof is the following Claim which provides us with the word $\jfi_0$. 

\begin{claim}\label{j_0}
There exists a finite word $\jfi_0$ and a collection $\cF \subseteq \cE$ such that $\cN(\jfi_0) = \cE$ and
\begin{itemize}
\item[(i)] for every $h \in \cF$, every neighbour of $K_{\jfi_0}$ is a descendant of $h(K)$, that is, there exists a word $\jfi_h \in \Gamma^*$ such that $h \circ \varphi_{\jfi_h} = \varphi_{\jfi_0}$ and $\cN(\jfi_h) = \cE$, 

\item[(ii)] for every $h \in \cE \setminus \cF$, $B_{\jfi_0} \cap h(K) = \emptyset$. 
\end{itemize}
\end{claim}

\begin{proof}
In the case where $K_{\ifi_0} \cap h(K) = \emptyset$ for every $h \in \cE \setminus \lbrace {\rm Id} \rbrace$ the claim is close to trivial: set $\cF = \lbrace {\rm Id} \rbrace$ and choose $\jfi_0 \in \Gamma^*$ in such a way that $B_{\jfi_0} \cap \bigcup_{h \in \cE \setminus \cF} h(K) = \emptyset$ as we may since the union is closed. 

Suppose then that $K_{\ifi_0} \cap h_1(K) \neq \emptyset$ for some $h_1 \in \cE \setminus \lbrace {\rm Id} \rbrace$. The set $\cF$ and the word $\jfi_0$ are now found through the following recursive construction. First, we claim that a descendant of $h_1(K)$ overlaps exactly with a descendant of $K$.

\begin{claim}\label{claim1}
There exist finite words $\jfi, \kfi \in \Gamma^*$ such that
$$
\varphi_{\jfi} = h_1 \circ \varphi_{\kfi}.
$$
\end{claim}

\begin{proof}[Proof of Claim \ref{claim1}]
Since we assume $K_{\ifi_0} \cap h_1(K) \neq \emptyset$, there exists $\kfi \in \Gamma^{|\ifi_0|}$ such that $h_1(K_{\kfi}) \cap K_{\ifi_0} \neq \emptyset$. Because $h_1 \in \cE = \cN(\ifi_0)$, by definition of $\cN$ there exists $\ifi \in \Gamma^{|\ifi_0|}$ such that $h_1 = \varphi_{\ifi_0}^{-1} \circ \varphi_{ \ifi}$, so we have 
$$
\varphi_{\ifi_0}^{-1} \circ \varphi_{ \ifi \kfi}(K) \cap \varphi_{\ifi_0}(K) \neq \emptyset. 
$$
Mapping both sets by $\varphi_{\ifi_0}$, we can use Lemma \ref{maximal} to deduce that $\varphi_{\ifi\kfi} \in \varphi_{\ifi_0 \ifi_0} \cE$, or equivalently, $h_1 \circ \varphi_{\kfi} \in \varphi_{\ifi_0}\cE$.

Since $\cE$ consists of functions of the form $\varphi_{\ifi_0}^{-1} \circ \varphi_\jfi$ for $\jfi \in \Gamma^{|\ifi_0|}$, there exists $\jfi$ such that $h_1 \circ \varphi_{\kfi} = \varphi_{\ifi_0} \circ \varphi_{\ifi_0}^{-1} \circ \varphi_{\jfi} = \varphi_{\jfi}$. The claim is satisfied with the words $\jfi$ and $\kfi$. 
\end{proof}

Write $\jfi_1 := \jfi$ and $\kfi_{1,1} := \kfi$ for the words $\jfi$ and $\kfi$ given by Claim \ref{claim1}. 

Suppose now that for an integer $n \geq 1$ and every $1 \leq k \leq n$ we have found functions $h_k$ and words $\jfi_n, \kfi_{n,k}$ such that $\varphi_{\jfi_n} = h_{k} \circ \varphi_{\kfi_{n,k}}$ for every $k$. Now, if there exists a function $h_{n+1} \in \cE \setminus \lbrace {\rm Id}, h_1, \ldots, h_n \rbrace$ for which $K_{\jfi_n\ifi_0} \cap h_{n+1}(K) \neq \emptyset$, proceed as in the proof of Claim \ref{claim1} with $\ifi_0$ replaced by $\jfi_n \ifi_0$ to find words $\jfi, \kfi$ such that $\varphi_{\jfi_n \jfi} = h_{n+1} \circ \varphi_{\kfi} = h_k \circ \varphi_{\kfi_{n,k} \jfi}$ for every $1 \leq k \leq n$; the latter equality follows from the hypothesis on $\jfi_n$ and $\kfi_{n, k}$. Set $\jfi_{n+1} := \jfi_n \jfi$, $\kfi_{n+1, n+1} := \kfi$ and $\kfi_{n+1, k} := \kfi_{n, k} \jfi$ for each $1 \leq k \leq n$. 

Continue this process until we have reached an integer $n < \# \cE$ for which either $\cE = \lbrace {\rm Id}, h_1, \ldots, h_n \rbrace$ or $K_{\jfi_n\ifi_0} \cap h(K) = \emptyset$ for each $h \in \cE \setminus \lbrace {\rm Id}, h_1, \ldots, h_n \rbrace$. Write $\cF = \lbrace {\rm Id}, h_1, \ldots, h_n \rbrace$ and note that by construction of $\cF$, for every $h \in \cF$ there exists $\jfi_h' \in \Gamma^*$ such that $\varphi_{\jfi_n} = h \circ \varphi_{\jfi_h'}$. By passing onto a descendant of $\jfi_n$ we may assume that $B_{\jfi_n} \cap \bigcup_{h \in \cE \setminus \cF} h(K)=\emptyset$ since the union is closed. Finally, set $\jfi_0 = \jfi_n \ifi_0$ and $\jfi_h = \jfi_h' \ifi_0$ for each $h \in \cF$ to obtain the lemma.

\end{proof}

\begin{proof}[Proof of Proposition \ref{nu_p}]
We first consider the claimed representation for symbolic magnifications of $\mu$. Using self-similarity, we first write
$$
(\varphi_{\io|_k}^{-1}\mu)|_{B(0,1)} = \varphi_{\io|_k}^{-1} \sum_{\ifi \in \Gamma^k,\ K_{\ifi} \cap B_{\io|_k} \neq \emptyset} p_{\ifi} \cdot (\varphi_{\ifi} \mu)|_{B(0,1)}.
$$
Using the equality $\lbrace \varphi_{\ifi}:\ \ifi \in \Gamma^k,\ K_{\ifi} \cap B_{\io|_k} \rbrace = \varphi_{\io|_k} \cE$ which is Lemma \ref{maximal}, we can write the above measure as
\begin{equation}\label{reorder}
\sum_{f \in \cE} \left( \sum_{\ifi \in \Gamma^k,\ \varphi_{\ifi} = \varphi_{\io|_k} \circ f} p_{\ifi} \right)\cdot (f \mu)|_{B(0,1)}.
\end{equation}
Recall that $k$ is such that $\sigma^{k-|\ifi_0 \jfi_0|} \io \in [\ifi_0 \jfi_0]$. Writing $k' = k - |\jfi_0|$, we further reorder the above sum using the following claim. It is a simple consequence of Claim \ref{j_0} and says that each neighbour of $K_{\io|_k}$ is a descendant of $\varphi_{\io|_{k'}} \circ h(K)$ if and only if $h \in \cF$.

\begin{claim}\label{partition}
For each $f \in \cE$, we may write
$$
\lbrace \ifi \in \Gamma^*:\ \varphi_{\ifi} = \varphi_{\io|_k} \circ f \rbrace = \bigcup_{h \in \cF} \lbrace \jfi \lfi \in \Gamma^*:\ \varphi_{\jfi} = \varphi_{\io|_{k'}} \circ h,\ h \circ \varphi_{\lfi} = \varphi_{\jfi_0} \circ f \rbrace,
$$
where the union is disjoint. In particular, note that $\jfi$ and $\lfi$ can be chosen independently of each other.
\end{claim}

\begin{proof}[Proof of Claim \ref{partition}]

First we observe that $\varphi_{\io|_k} \circ f(K)$ coincides with a descendant of $K_{\io|_{k'}}$. Indeed, because $\io|_k = \io|_{k'} \jfi_0$, this is equivalent to $\varphi_{\jfi_0} \circ f(K)$ coinciding with a descendant of $K$, which in turn follows by Lemma \ref{maximal} because $f \in \cE = \cN(\jfi_0)$. From this we can deduce that in particular, any ancestor of $\varphi_{\io|_k} \circ f(K)$ must intersect $B_{\io|_{k'}}$. Using this and the fact that $\varphi_i$ are similarities, we may write 
\begin{align*}
\lbrace \ifi:\ \varphi_{\ifi} = \varphi_{\io|_k}\circ f \rbrace  &= \lbrace (\jfi, \lfi) \in \Gamma^{k'} \times \Gamma^{|\jfi_0|}:\ K_{\jfi \lfi} = \varphi_{\io|_k} \circ f(K)\rbrace \\
&=\lbrace (\jfi, \lfi) \in \Gamma^{k'} \times \Gamma^{|\jfi_0|}:\ K_{\jfi} \cap B_{\io|_{k'}} \neq \emptyset,\ K_{\jfi \lfi} = \varphi_{\io|_k} \circ f(K) \rbrace.
\end{align*}
On the other hand, because $\cN(\io|_{k'})=\cE$, each neighbour of $K_{\io|_{k'}}$ coincides with $\varphi_{\io|_{k'}} \circ h(K)$ for some $h \in \cE$, by Lemma \ref{maximal}. Finally, by the properties of $\jfi_0$ given by Claim \ref{j_0}, as a neighbour of $K_{\io|_k}$ the set $\varphi_{\io|_{k}}\circ f(K)$ coincides with a descendant of $\varphi_{\io|_{k'}} \circ h(K)$ for each $h \in \cF$, and for $h \not\in \cF$, 
$$
B_{\io|_k} \cap \varphi_{\io|_{k'}} \circ h(K) = \varphi_{\io|_{k'}}(B_{\jfi_0} \cap h(K)) = \emptyset.
$$
Thus, we may write
\begin{align*}
&\lbrace (\jfi, \lfi) \in \Gamma^{k'} \times \Gamma^{|\jfi_0|}:\ K_{\jfi} \cap B_{\io|_{k'}} \neq \emptyset,\ K_{\jfi \lfi} = \varphi_{\io|_k} \circ f(K) \rbrace \\
=\ &\bigcup_{h \in \cF} \lbrace (\jfi, \lfi) \in \Gamma^{k'} \times \Gamma^{|\jfi_0|}:\ K_{\jfi} = \varphi_{\io|_{k'}} \circ h(K),\ \varphi_{\io|_{k'}} \circ h \circ \varphi_{\lfi}(K) = \varphi_{\io|_k} \circ f(K) \rbrace.
\end{align*}
Using the fact that $\varphi_i$ are similarities, the above can again be written as
\begin{align*}
&\bigcup_{h \in \cF} \lbrace \jfi \lfi\in \Gamma^*:\ \varphi_{\jfi} = \varphi_{\io|_{k'}} \circ h,\ \varphi_{\io|_{k'}} \circ h \circ \varphi_{\lfi} = \varphi_{\io|_{k'} \jfi_0} \circ f \rbrace \\
=\ &\bigcup_{h \in \cF} \lbrace \jfi \lfi\in \Gamma^*:\ \varphi_{\jfi} = \varphi_{\io|_{k'}} \circ h,\ h \circ \varphi_{\lfi} = \varphi_{\jfi_0} \circ f \rbrace,
\end{align*}
which is what was claimed.
\end{proof}

Using Claim \ref{partition}, we write the sum \eqref{reorder} as

\begin{align}\label{almostrepresentation}
&\sum_{f \in \cE} \left( \sum_{\ifi \in \Gamma^k,\ \varphi_{\ifi} = \varphi_{\io|_k}\circ f} p_{\ifi} \right) \cdot (f \mu)|_{B(0,1)}\nonumber\\
=\ &\sum_{h \in \cF} \sum_{f \in \cE} \left( \sum_{\jfi \lfi \in \Gamma^*,\ \varphi_{\jfi} = h\circ\varphi_{\io|_{k'}},\ h\circ\varphi_{\lfi} = \varphi_{\jfi_0}\circ f} p_{\jfi\lfi} \right) \cdot (f \mu)|_{B(0,1)}\nonumber\\
=\ &\sum_{h \in \cF} \tilde{q}_h(\io|_k)\sum_{f \in \cE} \left( \sum_{\lfi \in \Gamma^*,\ h\circ\varphi_{\lfi} = h \circ \varphi_{\jfi_h}\circ f} p_{\lfi}  \right) \cdot( f \mu)|_{B(0,1)}\nonumber \\
=\ &\sum_{h \in \cF} \tilde{q}_h(\io|_k) \sum_{f \in \cE} \left( \sum_{\lfi \in \Gamma^*,\ \varphi_{\lfi} = \varphi_{\jfi_h}\circ f} p_{\lfi} \right) \cdot (f \mu)|_{B(0,1)}
\end{align}
where 
$$
\tilde{q}_h(\io|_k) = \sum_{\jfi \in \Gamma^{k'},\ \varphi_{\jfi} = h \circ\varphi_{\io|_{k'}}} p_{\jfi}. 
$$
On the third row, we have summed over $\jfi \in \Gamma^{k'}$ for every fixed $\lfi$ and used the equality $\varphi_{\jfi_0} = h \circ \varphi_{\jfi_h}$ for every $h \in \cF$, given by Claim \ref{j_0}.

Define now the functions
$$
\zeta_h = \sum_{f \in \cE} \left( \sum_{\lfi \in \Gamma^*,\ \varphi_{\lfi} = \varphi_{\jfi_h}\circ f} p_{\lfi} \right) \cdot \frac{d (f\mu)|_{B(0,1)}}{d\nu}
$$
for $h \in \cF$, where the last expression denotes the Radon-Nikodym derivative of $(f\mu)|_{B(0,1)}$ with respect to $\nu$. Because $\cN(\jfi_h) = \cE$, for every $h \in \cF$ we have 
$$
\zeta_h(x) \geq C_h := \min_{f \in \cE} \sum_{\lfi \in \Gamma^*,\ \varphi_{\lfi} = \varphi_{\jfi_h}\circ f} p_{\lfi} >0 
$$
for $\nu$-a.e. $x \in \bigcup_{f \in \cE} f(K) \cap B(0,1)$. Normalizing the coefficients $\tilde{q}_h(\io|_k)$ in \eqref{almostrepresentation}, writing $(q_h(\io|_k))_{h \in \cF}$ for the obtained probability vector and defining $\zeta := \zeta(\io|_k) = \sum_{h \in \cF} q_h(\io|_k) \zeta_h$, we obtain from \eqref{almostrepresentation} the representation
\begin{equation}\label{normalize}
\mu_{\io|_k} = (\zeta d\nu)_{B(0,1)}.
\end{equation}

An elementary calculation shows that $S_{k\log \rho^{-1}} T_{\io} = T_{\sigma^k \io} \varphi_{\io|_k}^{-1}$. For any $t > 0$ for which $B(\pi(\io), e^{-t}) \subseteq B_{\io|_k}$, using \eqref{normalize} we have
\begin{align*}
\mu_{\io, t} &= S^*_{t + k\log \rho} S_{-k \log \rho} T_{\io} \mu \\
&= S_{t+k \log \rho}^* T_{\sigma^k \io} \mu_{\io|_k}\\
&= (\zeta d\nu)_{\sigma^k \io, t + k\log \rho}
\end{align*}
which is what was claimed.
\end{proof}

\subsection{Proof of Proposition \ref{distribution_asymptotic}}

Proposition \ref{distribution_asymptotic} is a consequence of the uniform continuity of the function $\zeta \mapsto (\zeta d\nu)_{\io, t}$ which is established in the following lemmas. 

The first one is a simple application of the Lebesgue-Besicovitch differentiation theorem.

\begin{lemma}\label{conv_to_0}
If $\eta$ and $\gamma$ are Radon measures and $0 \neq \eta \ll \gamma$, then for $\eta$-almost every $x$, 
$$
\lim_{t \to \infty} \Vert \eta_{x,t} - \gamma_{x,t} \Vert = 0.
$$
\end{lemma}

\begin{proof}
See \cite[Proposition 3.8]{Hochman}.
\end{proof}

\begin{comment}
Write $\phi = \frac{d\,\mu}{d\,\nu}$ and let $\cV_n$ be a partition of the image of $\phi$ into sets with diameter at most $1/n$. For each $V \in \cV_n$ with $\mu(\phi^{-1}V) > 0$, take the set of full $\mu$-measure $V' \subseteq \phi^{-1}V$ in which 
$$
\lim_{r \to 0} \frac{\mu(B(y,r) \cap \phi^{-1} V)}{\mu(B(y,r))} = \lim_{r \to 0} \frac{\nu(B(y,r) \cap \phi^{-1} V)}{\nu(B(y,r))} = 1
$$
for all $y \in V'$. Choose a point $x$ from the set $\bigcap_{n \in \N} \bigcup_{V \in \cV_n} V'$ of full $\mu$-measure. 

Now, let $\varepsilon> 0$ be arbitrary and choose a partition $\cV_n$ with $1/n < \delta$. Let $V \in \cV_n$ be such that $x \in V'$. For $\mu$-a.e. $y \in B(x,r) \cap \phi^{-1} V$, using the density theorem and definition of $\phi$, we have
$$
\lim_{t \to 0} \frac{\mu(B(y,t) \cap \phi^{-1}V)}{\nu(B(y,t))} = \phi(y) \in [\phi(x) - 1/n, \phi(x) + 1/n].
$$
Applying the Vitali covering theorem, for every $A \subseteq B(x,r)$, when $r$ is small we have
\begin{align*}
\mu(A) &= \mu(A \cap \phi^{-1}V)(1+o(r)) \\
&=\sum_i \mu(B(x_i, r_i) \cap \phi^{-1}V) \\
&\approx \sum_i \nu(B(x_i, r_i)) \\
&\approx \nu(A \cap \phi^{-1}V) \\
&\approx \nu(A).
\end{align*}
\end{proof}
\end{comment}

Recall that the set $\cW$ consists of all the functions $\zeta = \zeta(\io|_k)$ given by Proposition \ref{nu_p} for $\io \in \Gamma^\N$, $k \in \N$, and that $\zeta \geq C := \min_{h \in \cF} C_h > 0$ for each $\zeta \in \cW$. Write
$$
\cA = \bigcap_{\zeta \in \cW} \lbrace \io \in \Gamma^\N:\ \lim_{t \to \infty} \Vert \mu_{\io, t} - (\zeta d\nu)_{\io, t} \Vert = 0 \rbrace.
$$
Since $\mu \ll \zeta d\nu$ and $\cW$ is countable, we deduce from Lemma \ref{conv_to_0} that $\bar{\mu}(\cA) = 1$. In the following, the lower bound for the functions in $\cW$ comes to play: while the previous lemma asserts that $\mu_{\io, t}$ and $(\zeta d\nu)_{\io, t}$ have similar asymptotic behaviour in $t$, the following two lemmas show that this similarity is actually uniform over $\cW$, when $\io$ is given. 

\begin{lemma}\label{t_to_0}
For $\io \in \cA$ and $t \geq 0$, the function $f^{\io, t}: \cW \to \R$,
$$
\zeta \mapsto \Vert \mu_{\io, t} - (\zeta d\nu)_{\io,t} \Vert
$$
is continuous (when $\cW$ is equipped with the metric induced by the supremum norm) and the modulus of continuity is independent of $\io$ and $t$. 
\end{lemma}

\begin{proof}

Fix $\io$ and $t$. Let $\varepsilon > 0$ be given and let $\delta > 0$ be such that whenever $\zeta_1 ,\zeta_2 \in \cW$ and $\Vert \zeta_1 - \zeta_2 \Vert_\infty < \delta$, we have $\zeta_i \in [(1-\varepsilon) \zeta_j, (1 + \varepsilon) \zeta_j]$ for $i, j \in \lbrace 1, 2 \rbrace$. Choosing such a $\delta$ is possible because the functions $\zeta_1, \zeta_2$ are bounded from below by $C > 0$. 

Now, if $\Vert \zeta_1 - \zeta_2\Vert_\infty < \delta$, we have for $i, j \in \lbrace 1,2\rbrace$ and all measurable $A \subseteq B(0,1)$, 

\begin{align*}
(\zeta_i d\nu)_{\io, t}(A) - (\zeta_j d\nu)_{\io, t}(A) &= \frac{\int_{e^{-t}A + \pi(\io)} \zeta_i \,d\nu}{\int_{B(\pi(\io), e^{-t})} \zeta_i \,d\nu} - \frac{\int_{e^{-t}A + \pi(\io)} \zeta_j \,d\nu}{\int_{B(\pi(\io), e^{-t})} \zeta_j \,d\nu} \\
&\leq \frac{\int_{e^{-t}A + \pi(\io)} \left((1+\varepsilon)\zeta_i - (1-\varepsilon)\zeta_i\right) \,d\nu}{\int_{B(\pi(\io), e^{-t})} (1+\varepsilon)\zeta_i \,d\nu} \\
&\leq \frac{2\varepsilon}{1+\varepsilon}.
\end{align*}
Thus
$$
|f(\zeta_1, t, \io) - f(\zeta_2, t, \io)| \leq \Vert (\zeta_1 d\nu)_{\io, t} - (\zeta_2 d\nu)_{\io, t} \Vert < 2\varepsilon.
$$
Since $\varepsilon$ was arbitrary and $\delta$ depends only on $\varepsilon$ and $C$, this completes the proof.

\end{proof}

\begin{lemma}\label{uniform_convergence}
Let $X$ be a subset of a compact metric space and let $f$ be a function on $X \times \R$ such that $x \mapsto f(x,t)$ is continuous, with modulus of continuity independent of $t$, and for every $x$, $\lim_{t \to \infty} f(x,t) = 0$. Then for every $\varepsilon > 0$, 
$$
\sup_{x \in X} \inf \lbrace t_0:\ f(x,t) \leq \varepsilon\ \text{for every}\ t \geq t_0 \rbrace < \infty.
$$
\end{lemma}

\begin{proof}
Suppose otherwise, that there exists $c > 0$ with the property that for every $n \in \N$, there exists $x_n \in X$ and $t_n \geq n$ such that $f(x_n, t_n) > c$. Let $(y_n)_{n \in \N} \subseteq X$ be a Cauchy subsequence of $(x_n)_{n \in \N}$. Now, for every $\delta > 0$ there exist $N, M \in \N$ such that $|y_n - y_N| < \delta$, $f(y_N, t_n) < c/2$ and $f(y_n, t_n) > c$ for all $n \geq M$. Since the continuity of $x \mapsto f(x, t)$ was assumed to be uniform in $t$, this is a contradiction.
\end{proof}

We are now able to deduce Proposition \ref{distribution_asymptotic}. 

\begin{proof}[Proof of Proposition \ref{distribution_asymptotic}]

As a closed and bounded subset of the linear span of a finite family of functions, the closure of $\cW$ is compact. For $\io \in \cA$ and $t > 0$, let $f^{\io,t}$ denote the function of Lemma \ref{t_to_0}. Now, for each $\io \in \cA$ and $\varepsilon > 0$, apply Lemma \ref{uniform_convergence} for the function $(\zeta, t) \mapsto f^{\io, t}(\zeta)$ to obtain the existence of an integer $N(\io, \varepsilon) < \infty$ such that $\Vert \mu_{\io, t} - (\zeta d\nu)_{\io, t} \Vert < \varepsilon$ for every $\zeta \in \cW$ and $t \geq N(\io, \varepsilon)$. Since $\bar{\mu}( \bigcup_{N \in \N} \lbrace \io \in \cA:\ N(\io, \varepsilon) \leq N \rbrace ) = \bar{\mu}(\cA) =1$, we may choose an integer $N'$ so that if $\mathcal{J} = \lbrace \io \in \Gamma^\N:\ \Vert \mu_{\io, t} - (\zeta d\nu)_{\io, t}\Vert < \varepsilon\ \text{for all}\ t \geq N' \rbrace$, we have $\bar{\mu}(\cJ) > 0$. 

To obtain the statement of the proposition, choose $N$ large enough with respect to $N'$ that 
$$
d \left( \frac{1}{T} \int_0^T \delta[\mu_{\io, t}]\,dt,\ \frac{1}{T} \int_0^T \delta[(\zeta d\nu)_{\io, t}] \,dt \right) < \varepsilon
$$
for all $T \geq N$ and $\io \in \cJ$. That such an $N$ exists can be easily seen from the definition of the Prokhorov metric.
 
\end{proof}

\section{Discussion}\label{Discussion}	

We now recall the definition of a fractal distribution, and afterwards provide the short deduction of the second assertion of Theorem \ref{uniformly_scaling}.

Following the terminology of Hochman from \cite{Hochman}, \emph{fractal distributions} are distributions which are $S^*$-invariant and possess a spatial invariance property called \emph{quasi-Palm}. They generalize the notion of CP-distributions introduced by Furstenberg to a coordinate-free setting, and, in the context of scenery flows, the well-known principle of Preiss that ``tangent measures to tangent measures are tangent measures''.

\begin{definition}\label{fractaldistribution}
Let $P$ be a distribution on the space of probability measures on $B(0,1)$. We say that $P$ is \emph{quasi-Palm} if for any measurable $\cA$, $P(\cA) = 1$ if and only if for every $r > 0$, $P$-almost every $\nu$ satisfies
$$
\nu_{x,r} \in \cA
$$
for $\nu$-almost every $x$ with $B(x,e^{-r}) \subseteq B(0,1)$. 

A \emph{fractal distribution} is a distribution which is $S^*$-invariant and quasi-Palm. If the distribution is also ergodic with respect to $S^*$, we say that it is an \emph{ergodic fractal distribution}.
\end{definition}

\begin{remark}
This definition coincides with Hochman's definition of a \emph{restricted} fractal distribution in \cite{Hochman}, and this particular formulation was used by Hochman and Shmerkin in \cite{HS_equidistribution}. The definition of a fractal distribution in \cite{Hochman} considers distributions supported on all Radon measures of $\R^d$ whose support contains the origin, and does not require the condition $B(x, e^{-r}) \subseteq B(0,1)$. Since we only consider distributions arising from the scenery flow, the restricted versions are more relevant in our context. However, since it was shown in \cite{Hochman} that each restricted fractal distribution extends uniquely to a fractal distribution on the space of all Radon measures, the results of \cite{Hochman} that we use also apply with the above definition.
\end{remark}

We will now explain how to deduce the second assertion of Theorem \ref{uniformly_scaling} from the results of \cite{Hochman}. Suppose that $\mu$ is a self-similar measure satisfying the hypothesis of Theorem \ref{specialcase}: With the modifications highlighted in Subsection \ref{subsection}, the following goes through also in the setting of Theorem \ref{uniformly_scaling}.

Let $P$ denote the limit of the scenery flow of $\mu$. It is due to Hochman \cite[Theorem 1.7]{Hochman} that $P$ is a fractal distribution. It is not difficult to see from the representation \eqref{givesmass} that $P$ gives positive mass to the closed set 
\begin{equation}\label{set}
\lbrace (\zeta d\nu)_{\io, t}:\ \zeta \in \overline{\cW},\ \io \in \Gamma^\N,\ 0 \leq t \leq R \rbrace
\end{equation}
for large enough $R$. In particular, there exists an ergodic component $P'$ of $P$ which gives positive mass to the set \eqref{set}. Because ergodic components of fractal distributions are also fractal distributions by \cite[Theorem 1.3]{Hochman}, $P'$ is an ergodic fractal distribution. From the ergodic theorem and the quasi-Palm property it is not difficult to see that typical measures for $P'$ are uniformly scaling and generate $P'$. In particular, a measure in the set \eqref{set} generates $P'$. Finally, since each of these measures contains as an absolutely continuous component a translated and scaled copy of $\mu$, Lemma \ref{conv_to_0} allows us to deduce that the distribution generated by $\mu$ also equals $P'$. In particular, $\mu$ generates an ergodic fractal distribution.

\begin{remark}
It suffices to consider any accumulation point of the scenery flow in place of $P$ in the preceding paragraph. This way, after observing from \eqref{givesmass} that an accumulation point is supported on measures of the form \eqref{set}, one can infer another proof of Theorem \ref{uniformly_scaling} that relies on the deep results of \cite{Hochman} instead of the more hands-on geometric analysis of Proposition \ref{nu_p}.
\end{remark}

\subsection{Projections of Markov measures}

Theorem \ref{uniformly_scaling} extends for natural projections of ergodic Markov measures on $\Gamma^\N$. Write $\bar{\mu}$ for such a measure on $\Gamma^\N$ and $\mu:=\pi\bar{\mu}$ for its natural projection on the self-similar set. We go through the main observations one must make in addition to the ones in the Bernoulli case; details are left to the interested reader.

Let $\Sigma \subseteq \Gamma^\N$ denote the subshift of finite type associated to $\bar{\mu}$, and $\Sigma^*$ the collection of its finite words. First we note that $\#\lbrace \mu^{\ifi}:= \pi(\bar{\mu}_{[\ifi]}):\ \ifi \in \Sigma^* \rbrace \leq (\#\Gamma)^2 < \infty$ and that for any $\jfi \in \Sigma^*$, $\mu^{\ifi \jfi} = \mu^{\lfi \jfi}$ for all $\ifi, \lfi \in \Sigma^*$ for which $\ifi \jfi, \lfi \jfi \in \Sigma^*$. These are immediate from the Markov property of $\bar{\mu}$. The definition of the neighbourhood system of $\ifi \in \Sigma^*$ is replaced by
$$
\cN(\ifi) := \lbrace \varphi_{\ifi}^{-1} \circ \varphi_{\jfi}:\ \jfi \in \Sigma^*,\ \rho_{\jfi} \leq \rho_{\ifi} \leq \rho_{\jfi^-},\ K_{\jfi} \cap B_{\ifi} \neq \emptyset \rbrace.
$$
Arguing similarly as in the proof of Lemma \ref{maximal}, using the weak separation condition and finiteness of the collection $\cM := \lbrace \mu^{\ifi}:\ \ifi \in \Sigma^* \rbrace$, we see that there exists a word $\ifi_0' \in \Sigma^*$ such that for any $\jfi$ for which $\jfi \ifi_0' \in \Sigma^*$, $\cN(\jfi \ifi_0') = \cN(\ifi_0') =: \cE$ and 
\begin{equation}\label{Markovcondition}
\mu_{\jfi \ifi_0'} \sim \mu_{\ifi_0'}.
\end{equation} 
Recall \eqref{symbolicmag} for the definition of the symbolic magnifications $\mu_{\ifi}$. The equivalence \eqref{Markovcondition} follows from the condition $\cN(\jfi \ifi_0') = \cN(\ifi_0')$ when $\bar{\mu}$ is Bernoulli, and allows us to compare the scenery measures to the uniform ``frame'' $\nu := \mu_{\ifi_0'}$ as in the case of the Bernoulli measure. Now, if the collection $\cF \subseteq \cE$ and the words $\jfi_h, \jfi_0 \in \Sigma^*$ are as in Claim \ref{j_0}, we can go through the proof of Proposition \ref{nu_p} to obtain the same result with $\zeta(\io|_k)$ being a convex combination of the functions $0<\zeta_h^\eta = \frac{d\eta_{\jfi_h}}{d\nu}<\infty$, where $h \in \cF$ and $\eta \in \cM$ is such that $\eta_{\jfi_h} \neq 0$. The ergodic-theoretic work in Subsection \ref{ergodicpart} also goes through, as all properties of Bernoulli measures we use are shared by ergodic Markov measures. The ergodicity of the skew-product system we require in the rotating case follows from \cite[Theorem 2.2]{Kim}.

\subsection{Pointwise normality}
In \cite[Theorem 1.4]{HS_equidistribution}, Hochman and Shmerkin proved that if $\Phi = \lbrace \varphi_i(x) = \rho_i x + a_i \rbrace_{i \in \Gamma}$ is a self-similar IFS on the real line with the open set condition and if $s>1$ is a Pisot number such that $\frac{\log s}{\log \rho_i} \not\in \Q$ for some $i \in \Gamma$, then any self-similar measure associated to $\Phi$ is pointwise $s$-normal. The open set condition was required partly because of the uniform scaling assumption.

Combining Theorem \ref{uniformly_scaling} with the methods of Hochman-Shmerkin, we obtain the following:

\begin{corollary}\label{equidistribution}
Let $\Phi = \lbrace \varphi_i(x) = \rho_i x + a_i \rbrace_{i \in \Gamma}$ be a self-similar IFS on $\R$ satisfying the weak separation condition, and let $\mu$ be a self-similar measure associated to $\Phi$. If $s > 1$ a Pisot number and $\frac{\log s}{\log \rho_i} \not\in \Q$ for some $i \in \Gamma$, then $\mu$ is pointwise $s$-normal.
\end{corollary}

\begin{proof}
In the proof of \cite[Theorem 1.4]{HS_equidistribution}, the open set condition was only used to deduce that (i) $\mu$ is uniformly scaling and generates an ergodic fractal distribution $P$, and that (ii) given an eigenfunction of $P$, the phase measure of $\mu$ is an atom. Recall \cite[Section 4.3]{HS_equidistribution} for the definitions of an eigenfunction and a phase measure. Since (i) is true by Theorem \ref{uniformly_scaling}, we will only explain how to deduce that the phase measure of $\mu$ is an atom in the setting of Corollary \ref{equidistribution}; after this, the proof proceeds exactly as the proof of \cite[Theorem 1.4]{HS_equidistribution}.

It follows from \cite[Proposition 4.15]{HS_equidistribution} that given any eigenfunction, the phase measure of $P$-almost every $\eta$ is an atom. Since $P$ gives positive mass to the set \eqref{set}, it follows that for some $\zeta \in \overline{\cW}, \io \in \Gamma^\N$ and $t \geq 0$, the phase measure of $\eta = \zeta \,d\nu_{\io, t}$ is an atom. Since $\mu_{\io, t} \ll \eta$, using the self-similarity of $\mu$ we find a linear map $f$ such that $f\mu \ll \eta$. By \cite[Corollary 4.17, part 1]{HS_equidistribution}, the phase measure of $f\mu$ is also an atom, and by applying \cite[Corollary 4.17, part 2]{HS_equidistribution} to the image of $f\mu$ through $f^{-1}$, we deduce that also the phase measure of $\mu$ is an atom. 
\end{proof}

We remark that refinements of \cite[Theorem 1.4]{HS_equidistribution} have been obtained by Algom, Rodriguez Hertz and Wang \cite{AHW} for self-similar measures with the Rajchman property, and by Dayan, Ganguly and Weiss \cite{DGW} for certain self-similar measures defined by contractions with integer contraction ratios. While these works cover a large class of self-similar measures outside the assumptions of Hochman-Shmerkin, it is nevertheless not difficult to construct self-similar measures which satisfy the assumptions of Corollary \ref{equidistribution} but do not satisfy the assumptions of \cite[Theorem 1.4]{HS_equidistribution}, \cite{AHW} or \cite{DGW}. Consider, for example, the Bernoulli convolution $\nu_\rho$, i.e. the self-similar measure associated to the IFS $\lbrace x \mapsto \rho x -1, x\mapsto \rho x + 1\rbrace$ with the uniform probability vector. It is well-known that when $0 < \rho < 1$ is such that $\rho^{-1}$ is a Pisot number, $\nu_{\rho}$ is not a Rajchman measure \cite{Erdos} but it satisfies the weak separation condition \cite{Garsia, LN}. By Corollary \ref{equidistribution}, $\nu_\rho$ is pointwise $s$-normal for any Pisot $s> 1$ such that $\frac{\log s}{\log \rho} \not\in \Q$.

\subsection{Prospects}

\subsubsection{Weaker separation conditions}
In our proof of Theorem \ref{uniformly_scaling}, the weak separation condition played a crucial role in providing the reference measure $\nu$ in \eqref{nudefinition}. A natural question is whether the statement of the theorem holds in the presence of weaker separation conditions, for example, the \emph{asymptotic weak separation condition} introduced in \cite{Feng_asymptotic}. 
\begin{question}
Let $\Phi = \lbrace \varphi_i \rbrace_{i \in \Gamma}$ be a self-similar IFS satisfying the asymptotic weak separation condition, that is, 
$$
\lim_{n \to \infty} \frac{\log(\max_{\ifi \in \Gamma^n} \# \cN(\ifi))}{n} =0.
$$
Are self-similar measures associated to $\Phi$ uniformly scaling?
\end{question}
This condition lacks the existence of a maximal neighbourhood system, instead allowing the cardinalities of the neighbourhood systems to grow subexponentially in $n$. As a consequence, it is much more difficult to understand the dynamic behaviour of the sequence $(\cN(\io|_k))_{k \in \N}$ for any $\io \in \Gamma^\N$, and it seems that substantial refinements are required in our argument for one to be able to say anything about the scenery flow under this condition.

\subsubsection{A description of the tangent distribution}
In our proof of Theorem \ref{uniformly_scaling}, the tangent distribution of $\mu$ arises from a limiting argument. Therefore, it would be interesting to find a description of the tangent distribution for at least one explicit self-similar measure which satisfies the weak separation condition but not the open set condition. This would require describing the distribution of the sequence $(\zeta(\io|_k))_{k \in \N}$ in \eqref{givesmass} in more detail than what is required for the approximation \eqref{muestimate}.

\bibliographystyle{amsplain}
\bibliography{Scenery_flow_bibliography}

\providecommand{\bysame}{\leavevmode\hbox to3em{\hrulefill}\thinspace}
\providecommand{\MR}{\relax\ifhmode\unskip\space\fi MR }
\providecommand{\MRhref}[2]{%
  \href{http://www.ams.org/mathscinet-getitem?mr=#1}{#2}
}
\providecommand{\href}[2]{#2}
\begin{thebibliography}{10}

\bibitem{AHW}
Amir Algom, Federico Rodriguez~Hertz, and Zhiren Wang, \emph{Pointwise
  normality and fourier decay for self-conformal measures}, Preprint, available
  at https://arxiv.org/abs/2012.06529 (2020).

\bibitem{DGW}
Yiftach Dayan, Arijit Ganguly, and Barak Weiss, \emph{Random walks on tori and
  normal numbers in self-similar sets}, Preprint, available at
  https://arxiv.org/abs/2002.00455 (2020).

\bibitem{Erdos}
Paul Erd\"{o}s, \emph{On a family of symmetric {B}ernoulli convolutions}, Amer.
  J. Math. \textbf{61} (1939), 974--976. \MR{311}

\bibitem{Falconerbook}
K.~Falconer, \emph{Fractal geometry: Mathematical foundations and
  applications}, Wiley, 2013.

\bibitem{Feng_asymptotic}
De-Jun Feng, \emph{Gibbs properties of self-conformal measures and the
  multifractal formalism}, Ergodic Theory Dynam. Systems \textbf{27} (2007),
  no.~3, 787--812. \MR{2322179}

\bibitem{Fengpreprint}
\bysame, \emph{Uniformly scaling property of self-similar measures with the
  finite type condition},  (Unpublished manuscript).

\bibitem{FL}
De-Jun Feng and Ka-Sing Lau, \emph{Multifractal formalism for self-similar
  measures with weak separation condition}, J. Math. Pures Appl. (9)
  \textbf{92} (2009), no.~4, 407--428. \MR{2569186}

\bibitem{FFS}
Andrew Ferguson, Jonathan~M. Fraser, and Tuomas Sahlsten, \emph{Scaling scenery
  of {$(\times m,\times n)$} invariant measures}, Adv. Math. \textbf{268}
  (2015), 564--602. \MR{3276605}

\bibitem{FHOR}
J.~M. Fraser, A.~M. Henderson, E.~J. Olson, and J.~C. Robinson, \emph{On the
  {A}ssouad dimension of self-similar sets with overlaps}, Adv. Math.
  \textbf{273} (2015), 188--214. \MR{3311761}

\bibitem{FP}
Jonathan Fraser and Mark Pollicott, \emph{Uniform scaling limits for ergodic
  measures}, J. Fractal Geom. \textbf{4} (2017), no.~1, 1--19. \MR{3631374}

\bibitem{Furstenberg}
Hillel Furstenberg, \emph{Ergodic fractal measures and dimension conservation},
  Ergodic Theory Dynam. Systems \textbf{28} (2008), no.~2, 405--422.
  \MR{2408385}

\bibitem{Garsia}
Adriano~M. Garsia, \emph{Arithmetic properties of {B}ernoulli convolutions},
  Trans. Amer. Math. Soc. \textbf{102} (1962), 409--432. \MR{137961}

\bibitem{Gavish}
Matan Gavish, \emph{Measures with uniform scaling scenery}, Ergodic Theory
  Dynam. Systems \textbf{31} (2011), no.~1, 33--48. \MR{2755920}

\bibitem{HHR}
Kathryn Hare, Kevin Hare, and Alex Rutar, \emph{When the weak separation
  condition implies the generalized finite type condition}, Proceedings of the
  American Mathematical Society (2020), 1.

\bibitem{Hochman}
Michael Hochman, \emph{Dynamics on fractals and fractal distributions},
  Preprint, available at https://arxiv.org/abs/1008.3731 (2010).

\bibitem{Hochmanoverlaps}
\bysame, \emph{On self-similar sets with overlaps and inverse theorems for
  entropy}, Ann. of Math. (2) \textbf{180} (2014), no.~2, 773--822.
  \MR{3224722}

\bibitem{HochmanICM}
\bysame, \emph{Dimension theory of self-similar sets and measures}, Proceedings
  of the {I}nternational {C}ongress of {M}athematicians---{R}io de {J}aneiro
  2018. {V}ol. {III}. {I}nvited lectures, World Sci. Publ., Hackensack, NJ,
  2018, pp.~1949--1972. \MR{3966837}

\bibitem{HS}
Michael Hochman and Pablo Shmerkin, \emph{Local entropy averages and
  projections of fractal measures}, Ann. of Math. (2) \textbf{175} (2012),
  no.~3, 1001--1059. \MR{2912701}

\bibitem{HS_equidistribution}
\bysame, \emph{Equidistribution from fractal measures}, Invent. Math.
  \textbf{202} (2015), no.~1, 427--479. \MR{3402802}

\bibitem{Kac}
M.~Kac, \emph{On the notion of recurrence in discrete stochastic processes},
  Bull. Amer. Math. Soc. \textbf{53} (1947), 1002--1010. \MR{22323}

\bibitem{KR}
Antti K\"{a}enm\"{a}ki and Eino Rossi, \emph{Weak separation condition,
  {A}ssouad dimension, and {F}urstenberg homogeneity}, Ann. Acad. Sci. Fenn.
  Math. \textbf{41} (2016), no.~1, 465--490. \MR{3467722}

\bibitem{Kakutani}
Shizuo Kakutani, \emph{Random ergodic theorems and {M}arkoff processes with a
  stable distribution}, Proceedings of the {S}econd {B}erkeley {S}ymposium on
  {M}athematical {S}tatistics and {P}robability, 1950, University of California
  Press, Berkeley and Los Angeles, 1951, pp.~247--261. \MR{0044773}

\bibitem{Kim}
Hyun~Jung Kim, \emph{Skew product action}, Int. J. Contemp. Math. Sci.
  \textbf{1} (2006), no.~5-8, 205--211. \MR{2289027}

\bibitem{LN}
Ka-Sing Lau and Sze-Man Ngai, \emph{Multifractal measures and a weak separation
  condition}, Adv. Math. \textbf{141} (1999), no.~1, 45--96. \MR{1667146}

\bibitem{LNfinitetype}
\bysame, \emph{A generalized finite type condition for iterated function
  systems}, Adv. Math. \textbf{208} (2007), no.~2, 647--671. \MR{2304331}

\bibitem{LW}
Ka-Sing Lau and Xiang-Yang Wang, \emph{Iterated function systems with a weak
  separation condition}, Studia Math. \textbf{161} (2004), no.~3, 249--268.
  \MR{2033017}

\bibitem{Oneil}
Toby O'Neil, \emph{A measure with a large set of tangent measures}, Proc. Amer.
  Math. Soc. \textbf{123} (1995), no.~7, 2217--2220. \MR{1264826}

\bibitem{Shmerkinlq}
Pablo Shmerkin, \emph{On {F}urstenberg's intersection conjecture, self-similar
  measures, and the {$L^q$} norms of convolutions}, Ann. of Math. (2)
  \textbf{189} (2019), no.~2, 319--391. \MR{3919361}

\bibitem{Varju}
P\'{e}ter~P. Varj\'{u}, \emph{On the dimension of {B}ernoulli convolutions for
  all transcendental parameters}, Ann. of Math. (2) \textbf{189} (2019), no.~3,
  1001--1011. \MR{3961088}

\bibitem{Zerner}
Martin P.~W. Zerner, \emph{Weak separation properties for self-similar sets},
  Proc. Amer. Math. Soc. \textbf{124} (1996), no.~11, 3529--3539. \MR{1343732}

\end{thebibliography}

\end{document}